\documentclass[11pt,a4paper]{article}
\usepackage[colorlinks]{hyperref}
\usepackage[latin1]{inputenc}
\usepackage{amsmath}
\usepackage{amsfonts}
\usepackage{amssymb}
\usepackage{theorem}
\usepackage{subfig}
\usepackage{graphicx}
\usepackage{xspace}
\usepackage{tikz}
\usepackage{fullpage}
\providecommand{\figwidth}{\textwidth}
\newcommand{\figscale}{1.1}
\newcommand{\rP}{\ensuremath{\mathbb P}\xspace}

\newcommand{\poly}[1]{\ensuremath{\rP}^{#1}}

\newcommand{\ndg}[1]{| \kern -.25mm \|{#1}| \kern -.25mm \|}

\newcommand{\ltwo}[2]{\|{#1}\|_{#2}}

\newcommand{\el}{ T \in \mathcal{T} }
\newcommand{\ud}{\mathrm{d}}

\newcommand{\norm}[2]{\|#1\|_{#2}}
\newcommand{\snorm}[2]{|#1|_{#2}}
\newcommand{\dint}{\text{\rm int}}
\newcommand{\mbf}[1]{\mbox{\boldmath$\rm{#1}$}}
\newcommand{\mean}[1]{ \{#1\} }
\newcommand{\jump}[1]{  [#1]  }
\newcommand{\ujump}[1]{  \lfloor #1\rfloor  }

\newcommand{\qp}[1]{\ensuremath{\!\left({#1}\right)}}
\newcommand{\rR}{\ensuremath{\mathbb R}\xspace}
\newcommand{\reals}{\rR}
\renewcommand{\vec}[1]{\ensuremath{\boldsymbol{#1}}}
\newcommand{\Forall}{\:\forall\:}

\newcommand{\Foreach}{\quad\Forall}
\newcommand{\transpose}{{\boldsymbol\intercal}}   % transpose symbol
\newcommand{\cE}{\mathcal E}

\DeclareMathOperator{\diam}{diam}

\newtheorem{corollary}{Corollary}[section]
\newtheorem{lemma}[corollary]{Lemma}
\newtheorem{theorem}[corollary]{Theorem}

\newtheorem{remark}[corollary]{Remark}
\newtheorem{example}[corollary]{Example}
\newtheorem{assumption}[corollary]{Assumption}

\newcommand{\qed}{ \vspace{-0.5cm} \hfill $\Box$ }
\newenvironment{proof}[1][Proof.]{\begin{trivlist}
\item[\hskip \labelsep {\bfseries #1}]}{\end{trivlist}\qed}

\begin{document}
\title{Recovered Finite Element Methods}
\author{
  Emmanuil H.~Georgoulis\thanks{
   Department of Mathematics,
    University of Leicester,
    University Road,
    Leicester LE1 7RH,
    UK and School of Applied Mathematical and Physical Sciences, National Technical University of Athens, Zografou 15780, Greece
    {\tt{Emmanuil.Georgoulis@le.ac.uk}}.
    }\and  Tristan Pryer\thanks{
     Department of Mathematics and Statistics,
     University of Reading,
     Whiteknights,
     PO Box 220,
     Reading RG6 6AX,
     UK
  {\tt{T.Pryer@reading.ac.uk}}.
}}
\date{\today}

\maketitle

\begin{abstract}
\noindent
We introduce a family of Galerkin finite element methods which are constructed via recovery operators over element-wise discontinuous approximation spaces. This new family, termed collectively as \emph{recovered finite element methods (R-FEM)} has a number of attractive features over both classical finite element and discontinuous Galerkin approaches, most important of which is its potential to produce stable conforming approximations in a variety of settings. Moreover, for special choices of recovery operators, R-FEM produces the same approximate solution as the classical conforming finite element method, while, trivially, one can recast (primal formulation) discontinuous Galerkin methods.  A priori error bounds are shown for linear second order boundary value problems, verifying the optimality of the proposed method. Residual-type a posteriori bounds are also derived, highlighting the potential of R-FEM in the context of adaptive computations. Numerical experiments highlight the good approximation properties of the method in practice. A discussion on the potential use of R-FEM in various settings is also included.
\end{abstract}

\section{Introduction}
Galerkin procedures are extremely popular in numerical approximation of solutions to initial and/or boundary value problems for partial differential equations (PDEs). The most used families of Galerkin procedures are the finite element (FEM) and, more recently, discontinuous Galerkin (dG) families of methods. Roughly speaking, FEM are attractive for their simplicity and robustness, especially in structural mechanics and heat flow simulations, owing to their variational interpretation and origins; dG methods, on the other hand, are popular in fluid flow and fast convection/transport simulations, due to their superior numerical stability properties, stemming from the ability to incorporate general numerical flux functions seamlessly.

FEM typically incorporate (approximate) continuity/conformity of the state variable(s) and/or of some moments directly into the finite element space in order to imitate the respective properties of the underlying continuous problem. As a result, FEM's approximation capabilities have to be assessed for each choice of finite element spaces. At the other end of the spectrum, dG methods typically employ element-wise discontinuous approximation spaces whose approximation properties are clear; the continuity/conformity of the state variable(s) and/or of some moments is enforced weakly via numerical flux functions incorporated in the variational formulation of the dG method. Consequently, the approximation capabilities of dG schemes is only dependent on the ability to show C\'ea-type quasi-optimality results, i.e., relies on the structure of the respective weak formulation of the dG method. This enables the construction of stable dG methods in a number of settings, e.g., hyperbolic and/or convection-dominated problems \cite{reedhill,MR58:31918,MR88b:65109, MR92e:65128, hss}, and locking-free approximations for elasticity \cite{MR1886000,MR2027288} (see also \cite{MR1140646,MR1343077} for similar ideas in the context of classical non-conforming methods).

In an effort to combine the ``generic'' approximation-space capabilities of dG methods while retaining the attractive conformity or near-conformity of the discretizations (enjoyed normally in FEM), we introduce a family of Galerkin finite element-type methods which are constructed via (conforming or classical non-conforming) recovery operators applied to element-wise discontinuous approximation spaces. This family will be termed collectively as \emph{recovered finite element methods (R-FEM)} and has a number of attractive features over both FEM and dG approaches. 

More specifically, R-FEM combines completely discontinuous local finite element-type spaces, resulting, nonetheless, to conforming FEM-like approximations, or classical non-conforming ones, e.g., of Crouzeix-Raviart type. To fix ideas, let $\mathcal{E}:V_h\to \tilde{V}_h\cap H^1_0(\Omega)$ an operator mapping a \emph{discontinuous} piecewise polynomial space $V_h$ over a triangulation onto a space of \emph{continuous} piecewise polynomial space $\tilde{V}_h\cap H^1_0(\Omega)$ over the same or a finer triangulation; such \emph{recovery} operators $\mathcal{E}$ can be constructed locally, e.g., by (weighted) averaging of the nodal degrees of freedom \cite{KP,MR0400739}. We can now consider the method: find $u_h\in V_h$, such that
\[
\int_{\Omega}\nabla \mathcal{E}(u_h)\cdot \nabla \mathcal{E}(v_h)\ud x + s(u_h,v_h) = \int_{\Omega}f\mathcal{E}(v_h)\ud x, \quad\text{ for all } v_h\in V_h,
\]
to solve the Poisson problem with homogeneous essential boundary conditions, $f\in H^{-1}(\Omega)$, for suitable \emph{stabilization} $s(\cdot,\cdot):W_h\times W_h\to\mathbb{R}$, where $W_h\subset V_h$ such that $V_h=W_h \oplus (V_h\cap H^1_0(\Omega))$. We observe that, despite using element-wise discontinuous polynomial test space $V_h$, the method also produces simultaneously a conforming approximation $\mathcal{E}(u_h)$. We note that the above method yields, in general, different numerical solutions to those one would get by postprocessing standard dG approximations via the recovery operator $\mathcal{E}$. On the other hand, as we shall see below, we can also retrieve classical conforming FEM approximations from R-FEM, by making specific choices of recovery operators $\mathcal{E}$. Therefore, in a sense R-FEM is both a generalisation and a variant of known finite element methods.

The above basic example is intended to highlight a number of attractive features for R-FEM: conformity is not hard-wired in the approximation spaces and there is considerable flexibility in the particular choice of: (a) the recovery operator $\mathcal{E}$; (b) the finite element space  $\mathcal{E}(V_h)$; and, (c) the stabilisation $s$ used. An important property of R-FEM is the extreme flexibility in the choice and nature of degrees of freedom to be recovered locally (e.g., nodal values in one element, normal fluxes or higher moments in another, etc.) while \emph{avoiding} the often cumbersome proof of respective unisolvency of the global linear system. Moreover, a crucial practical attribute of R-FEM is that it can be implemented in a rather straightforward manner into existing conforming or non-conforming finite element implementations, as we shall discuss below.

The use of element-wise discontinuous polynomial spaces in standard dG methods is often found to increase the number of numerical degrees of freedom on a given mesh, without achieving better approximation per degree of freedom. Generally speaking, this is the case for low order approximation spaces. Recent work has shown that this is \emph{not} necessarily the case for higher order polynomial degrees and/or $hp$-version dG  methods \cite{DGpoly1,DGpoly2,DGpolyparabolic}, once appropriate choices of local basis are determined, e.g., using local basis of \emph{total} polynomial degree on box-type or, generally, polygonal/polyhedral elements. Since the use of total degree bases on such meshes necessitates relaxation of conformity requirements for the approximation spaces, dG methods have been used in this context as the underlying discretization. It is evident that R-FEM can be naturally extended to accept such reduced approximation spaces, once a suitable choice of recovery is constructed; we refer to the forthcoming work \cite{REMpoly} for the construction of R-FEM for general polygonal/polyhedral element shapes. Also, recalling that R-FEM is based on discontinuous approximation spaces also, R-FEM is expected to be able to produce stable conforming approximations in the context of convection-dominated problems. A numerical investigation in this direction is presented in the numerical experiments below.

The idea of variational methods involving some form of recovery, e.g., as (enhanced) gradient approximation, has appeared in various forms in the literature, especially in the context of dG methods, see, e.g., \cite{Lew,HDG,BuffaOrtner,Ern}. All these approaches, however, recover the gradient of the numerical solution, through the solution of additional local (or global) weak problems.  At the other end of the spectrum, the recent framework of virtual element methods (VEMs) \cite{virtual} applied to meshes with polygonal/polyhedral element shapes, use non-polynomial basis functions to ensure conformity, which are not computed fully in practice. Indeed, the local VEM spaces are constructed so that they contain a, typically discontinuous, polynomial subspace ensuring optimal approximation; the approximate solution is then computed over this subspace and over the mesh skeleton.

As a first work introducing the recovered finite element method, we confine ourselves to its definition for linear elliptic problems on simplicial and/or box-type triangulations (Section \ref{sec:rfem}). After discussing its relation on known, popular methods (Section \ref{sec:fem}), we prove a priori bounds (Section \ref{sec:apriori}) showing the optimality of the proposed method with respect to meshsize $h$, as well as residual-type a posteriori error bounds (Section \ref{sec:apost}). We also investigate carefully the lowest order case, whereby the underlying discontinuous space $V_h$ consists of element-wise constant functions recovered into conforming linear elements on a triangular mesh; the resulting method is shown to converge optimally in the energy norm. We investigate a number of implementation and conditioning issues for the proposed method, assessing its competitiveness against FEM and dG in terms of complexity (Section \ref{sec:compl}).  Numerical experiments highlighting the good approximation properties of the method in practice and comparisons against FEM and dG are performed in Section \ref{sec:numerics}. Finally, we conclude by discussing a number of currently developed and future extensions of the R-FEM framework to complex problems.

Throughout this work we denote the standard Lebesgue spaces by $L^p(\omega)$, $1\le p\le \infty$, $\omega\subset\mathbb{R}^d$, $d=2,3$,
with corresponding norms $\|\cdot\|_{L^p(\omega)}$;
the norm of $L^2(\omega)$ will be denoted by $\ltwo{\cdot}{\omega}$ for brevity. Let also
$H^s(\omega)$, be the Hilbertian Sobolev
space of index $s\in\mathbb{R}$ of real-valued functions defined on
$\omega\subset\mathbb{R}^d$, constructed via standard interpolation and/or duality procedures, along with the corresponding norm and seminorm
$\norm{\cdot}{s,\omega}$ and $\snorm{\cdot}{s,\omega}$, respectively. We also denote by $H^1_0(\omega)$ the space of functions in $H^1(\omega)$ with vanishing trace on $\partial \omega$.

For $\Omega$ a bounded open polygonal domain in $\mathbb{R}^d$, $d=2,3$,
with $\partial\Omega$ denoting its boundary, we consider the elliptic problem
\begin{equation}\label{pde}
  -\nabla\cdot A\nabla u=f\quad\text{in } \Omega,
\end{equation}
where $f\in H^{-1}(\Omega)$, for some uniformly positive definite diffusion tensor $A\in [L^{\infty}(\Omega)]^{d\times d}$.
For simplicity of the presentation only, we impose homogeneous essential boundary
conditions $u=0$ on $\partial\Omega$, although this is by no means an essential restriction on what follows. Moreover, we shall also briefly consider a linear convection-diffusion problem in Section \ref{sec:numerics} to highlight the versatility of the proposed framework.

\section{Recovered finite element method}
\label{sec:rfem}
%\subsection{Meshes, finite element spaces and trace operators}\label{mesh_assumptions}
Let $\mathcal{T}$ be a regular subdivision of $\Omega$ into
disjoint simplicial or box-type (quadrilateral/hexahedral) elements $\el$%; extensions to polygonal/polyhedral elements, possibly containing hanging nodes/edges, will be discussed later.
. We assume that the subdivision $\mathcal{T}$ is shape-regular (see, e.g., p.124 in
\cite{ciarlet}), that $\bar{\Omega}=\cup_{\el}\bar{T}$
and that the elemental faces are straight line (for $d=2$) or planar (for $d=3$) segments; these will be, henceforth, referred to as \emph{facets}. 
By $\Gamma$ we shall denote the union of all ($d-1$)-dimensional facets associated with the subdivision $\mathcal{T}$ including the boundary. Further, we set $\Gamma_{\dint}:=\Gamma\backslash\partial\Omega$.

For a nonnegative integer $r$, we denote the set of all polynomials of
total degree at most $r$ by $\mathcal{P}_r(T)$, while the set of all tensor-product polynomials on $T$ of degree at most $r$
in each variable is denoted by $\mathcal{Q}_r(T)$.
For $r \geq 1$, we consider the finite element space
\begin{equation}
 V_h^r  :=\{v\in L^2(\Omega):v|_{T}
    \in\mathcal{R}_{r}(T),\,\el\},
\end{equation}
where $\mathcal{R}_{r}(T)\in\{\mathcal{P}_{r}(T), \mathcal{Q}_{r}(T)\}$. We stress that, in this context, we can consider local bases in $\mathcal{P}_r(T)$ also for box-type elements; we shall return to this point below, cf., \cite{DGpoly1,DGpoly2,DGpolyparabolic}.

Further, let $T^+$, $T^-$ be two (generic) elements sharing a facet
$e:=\partial T^+\cap\partial T^-\subset\Gamma_{\dint}$ with respective outward normal unit vectors $\mbf{n}^+$ and $\mbf{n}^-$ on $e$. For a function $v:\Omega\to\mathbb{R}$ that may be discontinuous across $\Gamma_{\dint}$,
we set  $v^+:=v|_{e\subset\partial T^+}$, $v^-:=v|_{e\subset\partial T^-}$, and we define the jump by
\[ \jump{v}:=q^+\mbf{n}^++q^-\mbf{n}^-;\]
if $e\in \partial T\cap\partial\Omega$, we set $\jump{v}:=v^+\mbf{n}$. Also, we define $h_{T}:=\diam(T)$ and we collect them into the
element-wise constant function ${\bf h}:\Omega\to\mathbb{R}$,
with ${\bf h}|_{T}=h_{T}$, $\el$, ${\bf h}|_e=(h_T+h_{T'})/2$ for $e\subset\Gamma_{\dint}$ and ${\bf h}|_e=h_T$ for $e\subset\partial T\cap\partial\Omega$. We assume that the families of meshes considered in this work are locally
quasi-uniform.

For the definition of the proposed method, we require \emph{recovery operators} of the form 
\begin{equation}\label{recovery}
\mathcal{E}:V_h^r\to   H^1_0(\Omega)\cap V_h^s,
\end{equation}
 for some non-negative integer $r$, mapping element-wise discontinuous functions into functions in the solution space for the boundary value problem, for some $s,r\in\mathbb{N}\cup\{0\}$, or respective non-conforming recovery
\begin{equation}
\mathcal{E}:V_h^r\to  V_{nc},
\end{equation} for some classical non-conforming finite element space $V_{nc}$, e.g., Crouzeix-Raviart elements. Evidently, there is considerable flexibility in the choice of such recovery operators; indeed, various choices of $\mathcal{E}$ may give rise to a different methods. 

For recovery operator $\mathcal{E}:V_h^r\to \widetilde{V}$, with $\widetilde{V}\in \{H^1_0(\Omega)\cap V_h^s, V_{nc}\}$, we consider the \emph{recovered finite element method} reading: find $u_h\in V_h^r$ such that
\begin{equation}\label{rem}
B(w_h,v_h):=a( \mathcal{E}(w_h),\mathcal{E}(v_h)) + s_h(w_h,v_h) = \ell (\mathcal{E}(v_h)),\qquad\text{for all } v_h\in \widetilde{V},
\end{equation}
where
\begin{equation}\label{REM_bilinear}
a(w,v):= \sum_{T\in\mathcal{T}}\int_{T} A\nabla w\cdot \nabla v\,\ud x,\quad\text{for all } w,v\in \prod_{T\in\mathcal{T}}H^1(T),
\end{equation}
and
\[
\ell (v):=\int_{\Omega} f v\,\ud x,\quad\text{for all } v\in \prod_{T\in\mathcal{T}}H^1(T),
\]
with $s_h(\cdot,\cdot):V_h^r\times V_h^r\to \mathbb{R}$ a symmetric bilinear form, henceforth referred to as the \emph{stabilisation}; possible choices for $s_h$, resulting, in general, to different methods, will be given below. As we shall see below, specific choices of $\mathcal{E}$ and of $s_h$ can lead to the same recovered solutions produced by both classical FEM and dG methods. 

Recovery operators of the form \eqref{recovery} have appeared in various settings in the theory of finite element methods, e.g., \cite{MR0400739,MR1011446,MR1248895, KP,brenner,GHV}. They are typically used to recover a conforming function from a non-conforming one under minimal regularity requirements. 

Classical examples include the construction of the so-called Cl\'ement or Scott-Zhang operators \cite{MR0400739,MR1011446}, based on local averages of functions. Note that the classical Cl\'ement construction yields a quasi-interpolant, i.e., it does \emph{not} necessarily preserve continuous functions.

Another popular example is the nodal \emph{averaging operator} for which the following stability result was proven by Karakashian and Pascal in \cite{KP}.
\begin{lemma} \label{lemma_E1}
Let  $\mathcal{T}$ a locally quasi-uniform mesh. The operator $\mathcal{E}_s:V_h^s\to V_h^s\cap H^1_0(\Omega)$, defined on the conforming Lagrange nodes $\nu\in\mathcal{N}$, $\mathcal{N}$ denoting the set of all Lagrange nodes of $V_h^s$, by:
 \[
 \mathcal{E}_s(v)(\nu):= \bigg \{ \begin{array}{cc}
 \displaystyle|\omega_{\nu}|^{-1}\sum_{T\in\omega_{\nu}} v|_{T}({\nu}),& \nu\in\Omega; \\ 
 0, & \nu\in \partial\Omega,
 \end{array} 
 \]
with
 $
 \omega_{\nu}:=\bigcup_{T\in\mathcal T: \nu\in\bar{T}}T, 
$
the set of elements sharing the node $\nu\in\mathcal{N}$ and $|\omega_{\nu}|$ their number. Then, the following bound  holds
\begin{equation}\label{KP_stab}
\sum_{T\in\mathcal{T}} \snorm{v - \mathcal{E}_s(v)}{\alpha,T}^2 \leq C_{KP,|\alpha|} \norm{\mbf{h}^{1/2-\alpha}\jump{v}}{\Gamma}^2,
\end{equation}
with  $|\alpha|=0,1$, $C_{|\alpha|}\equiv C_{|\alpha|}(r)>0$ a constant independent of $\mbf{h}$, $v$ and $\mathcal{T}$, but depending on the shape-regularity of $\mathcal{T}$ and on the polynomial degree $s$.
\end{lemma}

\begin{proof} See Karakashian and Pascal \cite{KP}.
\end{proof}

Note that, since $V_h^r\subset V_h^s$ for $r\le s$, the recovery operator $\mathcal{E}_s$ from Lemma \ref{lemma_E1} is also an operator from $V_h^r$ into $V_h^s\cap H^1_0(\Omega)$ for $0\le r\le s$. 
The bound \eqref{KP_stab} shows, in particular,  that $\norm{\mbf{h}^{-1/2}\jump{v}}{\Gamma}^2$ is a norm on the orthogonal complement $W_h^s$ of $V_h^s\cap H^1_0(\Omega)$ in $V_h^s$ with respect to the $a(\cdot,\cdot)$ inner product. This motivates the following choice for the stabilisation bilinear form:
\begin{equation}\label{stabilisation}
  s_h(w_h,v_h):= \int_{\Gamma}\sigma \jump{w_h}\cdot\jump{v_h}\,\ud s,
\end{equation}
for some non-negative function $\sigma:\Gamma\to\mathbb{R}$, to be defined below, henceforth referred to as \emph{discontinuity-penalization parameter}. In particular, since R-FEM \eqref{rem} is defined on the discontinuous space $V_h^r$, one has to take into account the part of $V_h^r$ that is not in $V_h^s\cap H^1_0(\Omega)$: Lemma \ref{lemma_E1} suggests that, upon defining $s_h$ is in \eqref{stabilisation}, the well-posedness of the method should be expected. Indeed, setting $w_h=v_h=w$ in \eqref{rem} we immediately arrive at the coercivity identity 
\begin{equation}
  \label{eq:coercivity-ident}
  \norm{\sqrt{A}\nabla \mathcal{E}(w)}{\Omega}^2 +\norm{\sqrt{\sigma}\jump{w}}{\Gamma}^2 =B(w,w) 
\quad \forall \  w\in V_h^r.
\end{equation}
With the help, now, of Lemma \ref{lemma_E1}, we can show that the left-hand side of \eqref{eq:coercivity-ident} is a norm in $V_h^r$. To this end, upon checking that $\norm{\sqrt{A}\nabla \mathcal{E}(w)}{\Omega}^2 +\norm{\sqrt{\sigma}\jump{w}}{\Gamma}^2=0$ implies $w=0$ (the other two norm properties being obvious). Indeed, we then have $\nabla\mathcal{E}(w)=0$ in $\Omega$ and $\jump{w}=0$ on $\Gamma$. Since $\mathcal{E}(w)\in H^1_0(\Omega)$, this,  in turn, implies $\mathcal{E}(w)=0$. Also, from Lemma \ref{lemma_E1}, for $\alpha =0$, we have
$w-\mathcal{E}(w)=0$, ensuring, thus, $w=0$.

\begin{remark}
The above argument may also motivate an alternative choice for the stabilisation $s$. In particular, choosing $s_h\equiv \widetilde{s_h}:V_h^r\times V_h^r\to\mathbb{R}$ with
\begin{equation}\label{stabilisation2}
  \widetilde{s_h}(w_h,v_h):=\int_{\Omega}\widetilde{\sigma}\big(w_h-\mathcal{E}(w_h)\big) \big(v_h-\mathcal{E}(v_h)\big)\,\ud x,
\end{equation}
for some non-negative function $\widetilde{\sigma}:\Gamma\to\mathbb{R}$, as we we, then, have the coercivity identity
\begin{equation}
  \norm{\sqrt{A}\nabla \mathcal{E}(w)}{\Omega}^2+\norm{\sqrt{\widetilde{\sigma}}\big(w-\mathcal{E}(w)\big)}{\Omega}^2 = B(w,w) \quad \forall \  w\in V_h^r. 
\end{equation}
\end{remark}

To retain a level of generality in the R-FEM framework presented below, we make the following assumptions on the bilinear form $s_h$, rather than prescribe it specifically.

\begin{assumption}\label{stabilisation_bound}
We assume that the stabilisation $s_h\equiv s_{h,\alpha}$ satisfies the equivalence
\begin{equation}\label{stab_equivalence}
 c_{0} \norm{\mbf{h}^{\alpha-1/2}\jump{v}}{\Gamma}^2\le  s_h(v,v)\le C_0\norm{\mbf{h}^{\alpha-1/2}\jump{v}}{\Gamma}^2,\quad\text{ for all } v\in V_h^r,
\end{equation} for some constant $\alpha\in \mathbb{R}$, and for $c_0,C_{0}>0$, depending only on the local elemental polynomial degree and on the mesh regularity, topology and geometry, but not on the local mesh-size. (This means that any dependence on the local mesh-size $\mbf{h}$ will be implicit in the definition of $s_h$ itself.)
\end{assumption} 
Note that both choices of stabilisation \eqref{stabilisation} and \eqref{stabilisation2} satisfy Assumption \ref{stabilisation_bound} with $\alpha =0$, while \eqref{stabilisation} also satisfies $c_0=C_0=1$.

\begin{assumption}\label{stabilisation_ass2}
We assume that the stabilisation $s_h$ satisfies the bound
\begin{equation}\label{stab_equivalence2}
s_h(w,v) \leq \widetilde{C}_{0} \big(s_h(w,w)\big)^{1/2}\big(s_h(v,v)\big)^{1/2},\quad\text{ for all } w,v\in V_h^r,
\end{equation}  for $\widetilde{C}_{0}>0$, depending only on the local elemental polynomial degree and on the mesh regularity, topology and geometry, but not on the local mesh-size.
\end{assumption} 
It is evident that both choices of stabilisation \eqref{stabilisation} and \eqref{stabilisation2} satisfy Assumption \ref{stabilisation_ass2} with $\widetilde{C}_0=1$.

\begin{remark}\label{remark_local}
In the proof of a posteriori error estimates below, we shall make use of local versions of Assumptions \ref{stabilisation_bound} and \ref{stabilisation_ass2},  applied to subsets $\omega\subset\Omega$ of the computational domain. These can be easily shown to be valid by using the bilinearity of $s_h$ and selecting $supp(v)\subset \omega$, noting that \eqref{stab_equivalence} implies that $s_h(v,v)$ will vanish away from $\omega$.
\end{remark}

\section{Connections to known methods}\label{sec:fem}

Interestingly, employing the nodal averaging recovery operator $\mathcal{E}_s$ in the R-FEM formulation \eqref{rem} on finite element spaces containing conforming subspaces, results to $\mathcal{E}_s(u_h)$ being the approximate solution computed via classical conforming finite element method! Indeed, the boundary value problem in weak form reads: find $u\in H^1_0(\Omega)$ such that $a(u,v)=\ell (v)$ for all $v\in H^1_0(\Omega)$.
Setting $v=\mathcal{E}_s(v_h)$ for a $v_h\in V_h^r$, we deduce
$a(u,\mathcal{E}_s(v_h))=\ell (\mathcal{E}_s(v_h))$,
which, upon subtraction of \eqref{rem}, leads to
\begin{equation}\label{Go}
a(u-\mathcal{E}_s(u_h),\mathcal{E}_s(v_h))=s_h(u_h,v_h), \quad \text{for all } v_h\in V_h^r.
\end{equation} Now, if $V_h^r$ is such that it contains the respective conforming subspace of the same order, setting $v_h\in V_h^r\cap H^1_0(\Omega)$, we have from \eqref{Go}
\begin{equation}\label{gabriel}
a(u-\mathcal{E}_s(u_h),v_h)=0, \quad \text{for all } v_h\in V_h^r\cap H^1_0(\Omega),
\end{equation}
noting, furthermore, that $v_h=\mathcal{E}_s(v_h)$ for $v_h\in V_h^r\cap H^1_0(\Omega)$. Therefore, \eqref{gabriel} implies that $\mathcal{E}_s(u_h) = u_h^{FEM}$ where, $u_h^{FEM}\in V_h^r\cap H^1_0(\Omega)$ is such that
\begin{equation}\label{FEM}
a(u_h^{FEM},v_h)=\ell(v_h), \quad \text{for all } v_h\in V_h^r\cap H^1_0(\Omega).
\end{equation}
Thus, R-FEM for special choices of recovery operators $\mathcal{E}$ and of finite dimensional spaces $V_h^r$ can retrieve the classical FEM solutions (albeit in a wasteful fashion). This is the content of the following result, whose proof is essentially included in the above discussion.
\begin{lemma}\label{RFEMeqFEM}
Let $V_h^r$ over a triangulation is such that it contains the conforming subspace $V_{h,c}^r:= V_h^r\cap H^1_0(\Omega)$ of the same polynomial degree. Assume also that the recovery operator used in \eqref{rem} is such that it preserves conforming functions, i.e., we have $\mathcal{E}(v_h)=v_h$, for all $v_h\in V_{h,c}^r$ and assume also that the stabilisation $s$ is such that $s(w_h,v_h)=0$. Then, $\mathcal{E}(u_h) =u_h^{FEM} $, when $u_h\in V_h^r$ is computed via \eqref{rem} and $u_h^{FEM}\in V_{h,c}^r$ is computed via \eqref{FEM}.
\end{lemma}
We shall be explicit in the use any of the assumptions of Lemma \ref{RFEMeqFEM} in the a priori and a posteriori analysis below. 

At the other end of the spectrum, we can trivially obtain interior penalty discontinuous Galerkin methods by setting $\mathcal{E}$ to be the identity operator, i.e., \emph{no} recovery and letting
\begin{equation}\label{dg_stab}
s_h(w_h,v_h)= \int_{\Gamma} \Big(\sigma \jump{w_h}\cdot\jump{v_h}-\mean{A\nabla w_h}\cdot\jump{v_h}-\theta\mean{A\nabla v_h}\cdot\jump{w_h}\Big)\ud s,
\end{equation}
with $\mean{\cdot}$ denoting the average operator on $\Gamma$, defined face-wise as $\mean{q_h}|_e:=\frac{1}{2}(q_h|_{e\subset \partial T}+q_h|_{e\subset \partial T'})$ for $e=\partial T\cap \partial T'$ and $T,T'$ neighbouring elements in $\mathcal{T}$; $\theta\in [-1,1]$ typically, and $\sigma$ denotes the usual discontinuity-penalization parameter. Note that Assumptions \ref{stabilisation_bound} and \ref{stabilisation_ass2} are not satisfied verbatim in this case. However, provided the discontinuity penalization parameter $\sigma$ is chosen large enough, see, e.g., \cite{arnold,unified,hss}, it can be shown that
\[
c_{0} \norm{\mbf{h}^{\alpha-1/2}\jump{v}}{\Gamma}^2-\frac{1}{2}\norm{\sqrt{A}\nabla\mathcal{E}(v)}{\Omega}^2\le  s_h(v,v)\le C_0\norm{\mbf{h}^{\alpha-1/2}\jump{v}}{\Gamma}^2+\frac{1}{2}\norm{\sqrt{A}\nabla\mathcal{E}(v)}{\Omega}^2,\text{ for all } v\in V_h^r,
\] and, respectively, for Assumption \ref{stabilisation_ass2}, allowing for the error analysis below to still hold. We refrained from using this, more general, version of Assumptions \ref{stabilisation_bound} and \ref{stabilisation_ass2} in this work for simplicity of the presentation, as the benefit of generality does not seem to provide any significant new insight at this point.

Crucially, however, R-FEM offers significant flexibility in the choice of both the finite element spaces, of the recovery operators and of the stabilisation terms, thereby allowing also for \emph{new} numerical methods.  Indeed, many classical recovery operators, e.g., of Cl\'ement type do \emph{not} satisfy the  condition $\mathcal{E}(v_h)=v_h$, for all $v_h\in V_{h,c}^r$, thereby giving rise to R-FEM with various properties even for the model elliptic problem considered above. A pertinent example will be given below where we consider recoveries from element-wise constant functions onto conforming linear elements or, in general, recovery between \emph{different} finite element spaces. Furthermore, as we shall see below, for more general PDE problems, such as convection-diffusion equations the flexibility offered by the R-FEM construction allows for the construction stable methods in the convection-dominating regime. In this case, the use of $\mathcal{E}_s$ can result in non-standard/novel numerical methods.

\section{A priori error analysis}
\label{sec:apriori}
The boundary value problem in weak form reads: find $u\in H^1_0(\Omega)$ such that $a(u,v)=\ell (v)$ for all $v\in H^1_0(\Omega)$.
Setting $v=\mathcal{E}(v_h)$ for a $v_h\in V_h^r$, we deduce
$a(u,\mathcal{E}(v_h))=\ell (\mathcal{E}(v_h))$,
which, upon subtraction of \eqref{rem}, leads to
\begin{equation}\label{Go}
a(u-\mathcal{E}(u_h),\mathcal{E}(v_h))=s_h(u_h,v_h), \quad \text{for all } v_h\in V_h^r.
\end{equation}
Using \eqref{Go} and Assumption \ref{stabilisation_ass2},
we have, respectively,
\begin{equation}\label{cea_proof}
\begin{aligned}
\norm{\sqrt{A}\nabla (u-\mathcal{E}(u_h))}{\Omega}^2 +s_h(u_h,u_h)
= \ & a(u-\mathcal{E}(u_h), u-\mathcal{E}(u_h))+s_h(u_h,u_h) \\
= \ & a(u-\mathcal{E}(u_h), u-\mathcal{E}(v_h))+s_h(u_h,v_h) \\
\le\ &\norm{\sqrt{A}\nabla (u-\mathcal{E}(u_h))}{\Omega}
\norm{\sqrt{A}\nabla(u-\mathcal{E}(v_h))}{\Omega} \\
&+\widetilde{C}_0\big(s_h(u_h,u_h)\big)^{1/2} \big(s_h(v_h,v_h)\big)^{1/2},
\end{aligned}
\end{equation}
for all $v_h\in V_h^r$, with the last step following from the Cauchy-Schwarz inequality and Assumption \ref{stabilisation_ass2}. This implies the quasi-optimality bound
\begin{equation}\label{cea}
\norm{\sqrt{A}\nabla (u-\mathcal{E}(u_h))}{\Omega}^2 +s_h(u_h,u_h)
\le \inf_{v_h\in V_h^r}\Big(\norm{\sqrt{A}\nabla(u-\mathcal{E}(v_h))}{\Omega}^2 +\widetilde{C}_0^2s_h(v_h,v_h)\Big).
\end{equation}

For $\Pi_r:L^2(\Omega)\to V_h^r$ denoting the orthogonal $L^2$-projection operator onto the finite dimensional space $V_h^r$, we have the estimate
\begin{equation}\label{stab_est}
\begin{aligned}
&\quad\inf_{v_h\in V_h^r}s_h(v_h,v_h) \le  s_h(\Pi_r u, \Pi_r u) \le C_0 \norm{\mbf{h}^{\alpha-1/2}\jump{\Pi_r u}}{\Gamma}^2
=C_0 \norm{\mbf{h}^{\alpha-1/2}\jump{u-\Pi_r u}}{\Gamma}^2 \\
&\le C  \sum_{T\in\mathcal{T}} \norm{\mbf{h}^{\alpha-1}(u-\Pi_r u)}{T}^2 
+\norm{\mbf{h}^{\alpha}\nabla(u-\Pi_r u)}{T}^2
\le  C  \sum_{T\in\mathcal{T}} h_T^{2(\alpha+s)} |u|_{s+1,T}^2,
\end{aligned}
\end{equation}
for $0\le s\le \min\{r,l\}$, when $u\in \prod_{T\in\mathcal{T}} H^{l+1}(T)\cap H^1_0(\Omega)$, for $l\ge 0$, using the standard trace estimate and the best approximation properties of $\Pi_r$.

Now, using Assumption \ref{stabilisation_bound} and \eqref{stab_est} on \eqref{cea} already proves the following result.

\begin{theorem}\label{thm:cea}
Let $u\in \prod_{T\in\mathcal{T}} H^{l+1}(T)\cap H^1_0(\Omega)$, for $l\ge 0$ be the solution to \eqref{pde} and $u_h\in V_h^r$ its R-FEM approximation with the stabilisation term satisfying Assumptions \ref{stabilisation_bound} and \ref{stabilisation_ass2}. Then, we have the a priori bound
\[
\quad\norm{\sqrt{A}\nabla (u-\mathcal{E}(u_h))}{\Omega}^2 
+c_0\norm{\mbf{h}^{\alpha-1/2}\jump{u_h}}{\Gamma}^2
\le \inf_{v_h\in V_h^r}\norm{\sqrt{A}\nabla(u-\mathcal{E}(v_h))}{\Omega}^2 +C  \sum_{T\in\mathcal{T}} h_T^{2(\alpha+s)} |u|_{s+1,T}^2,
\]
for all $0\le s\le \min\{r,l\}$, with $C$ a positive constant, independent of $u,$ $u_h$, $\mbf{h}$ and of $\mathcal{E}$.
\end{theorem}

Assuming now that the recovery operator $\mathcal{E}$ in the definition of R-FEM allows for optimal approximation of $u$ by $\mathcal{E}(v_h)$ for some $v_h\in V_h^r$, one can recover optimal a priori error bounds. Below we highlight some interesting cases of $\mathcal{E}$ and polynomial order $r$.

\subsection{Case $r=s\ge 1$.}\label{sec:reqs} We begin by considering the case of recovery into a conforming finite element space of the same polynomial degree. As discussed in Lemma \ref{RFEMeqFEM}, for certain choices of $\mathcal{E}$ this case results to $\mathcal{E}(u_h)$ being the conforming FEM solution. Nonetheless, as we shall see below, this is not necessarily the case for more general elliptic operators involving lower order terms,
(see Section \ref{sec:conv-diff} below,) or, indeed when the recovery takes place on a different triangulation than the underlying mesh $V_h^r$ is defined on; see Section \ref{ext:poly} below.

When $r=s$, the discontinuous space $V_h^r$ contains sufficient approximation power to ensure $r$-th order convergence in the energy norm.

\begin{corollary}\label{thm:reqs}
Let the discontinuity-penalization parameter $\sigma$ be given by \begin{equation}\label{sigma_reqs}
\sigma = c_\sigma \mathcal{A}\mathbf{h}^{-1},
\end{equation}
for any $c_\sigma>0$ and let the recovery operator $\mathcal{E}$ be such that $V_h^r\cap H^1_0(\Omega)\subset\mathcal{E}(V_h^r)$. Further, assuming that $u|_T\in H^{k+1}(T)$, $T\in\mathcal{T}$, $k\ge1$, we have the a priori bound
\begin{equation}\label{apriori_reqs}
\norm{\sqrt{A}\nabla (u-\mathcal{E}(u_h))}{\Omega}^2 +s_h(u_h,u_h)
\le C\sum_{T\in\mathcal{T}} \mathcal{A}|_{T} h_T^{2q}|u|_{q+1,T}^2,
\end{equation}
with $0\le q\le \min\{k,r\}$, for a $C$ positive constant, independent of the meshsize and of $u$.
\end{corollary}\qed

\subsection{Case $0\le r<s$.} 
One may wish to recover from an element-wise discontinuous polynomial space of order $r$, into an $H^1$-conforming polynomial space of strictly higher order $s$. In this setting, a recovered element method of optimal order \emph{at least} with respect to $r$ is expected. The analysis presented in Section \ref{sec:reqs} above, with the choice \eqref{sigma_reqs} for the discontinuity-penalization parameter, immediately gives the error bound
\eqref{apriori_reqs} with $0\le q\le \min\{k,r\}<s$. It is clear that the above bound is of lower order than the respective finite element method involving polynomials of order $s$. This is, perhaps, not surprising as the approximation power of $V_h^r$ allows, in general, for optimal rates with respect to $r$ only.

However, in practice there are some advantages in recovering from lower into higher polynomial order spaces. In particular, as we shall see below when using $r=s-1$, for $s=1,2,\dots$, with appropriately chosen power $\alpha$ in the stabilisation $s_h$, one can recover $s$-th order convergent method on simplicial meshes! To highlight this, somewhat surprising feature, we focus on the case $0=r<s=1$, i.e., recovering element-wise constants $u_h\in V_h^0$ into conforming element-wise linear elements $\mathcal{E}(u_h)\in V_h^1\cap H^1_0(\Omega)$, where first order convergence in the energy norm and second order convergence in the $L^2$-norm is observed for a suitable choice of the discontinuity-penalization parameter on shape-regular meshes; see Section \ref{sec:numerics} for a numerical illustration.

To show this, setting 
\begin{equation}\label{sigma_rneqs}
\sigma= c_\sigma \mathcal{A}\mathbf{h},
\end{equation}
for any $c_\sigma>0$, (and, consequently we shall have $\alpha=1$ in Assumption \ref{stabilisation_bound},) we deduce from Theorem \ref{thm:cea} that
\begin{equation}\label{cea_rneqs}
\norm{\sqrt{A}\nabla (u-\mathcal{E}(u_h))}{\Omega}^2 +s_h(u_h,u_h)
\le \inf_{v_h\in V_h^0}\norm{\sqrt{A}\nabla(u-\mathcal{E}(v_h))}{\Omega}^2+C\sum_{T\in\mathcal{T}} h_T^{2}|u|_{1,T}^2.
\end{equation}

If one can show, e.g., that $\mathcal{E}(V_h^0)= V_h^1\cap H^1_0(\Omega)$, (at least for some appropriately designed meshes,) then the first term on the right-hand side of \eqref{cea_rneqs} becomes
\begin{equation}\label{cea_low_eqs}
\inf_{v_h\in V_h^0}\norm{\sqrt{A}\nabla(u-\mathcal{E}(v_h))}{}^2=\inf_{w_h\in V_h^1\cap H^1_0(\Omega)}\norm{\sqrt{A}\nabla(u-w_h)}{}^2
\le C\sum_{T\in\mathcal{T}} h_T^{2}|u|_{T,2}^2,
\end{equation}
from standard Bramble-Hilbert type approximation results \cite{ciarlet},
i.e., optimal \emph{linear} convergence is retrieved. We shall now complete the proof of \eqref{cea_low_eqs} for the case of a two-dimensional simplicial mesh.

\begin{lemma} For $d=2$, let $\mathcal{T}$ be a regular simplicial mesh. Then, we have $\mathcal{E}(V_h^0)= V_h^1\cap H^1_0(\Omega)$.
\end{lemma}
\begin{proof} 
Recalling that $\mathcal{N}$ denotes the set of all Lagrange nodes $\nu\in\Omega$, we set $\mathcal{N}_0$ to be the set of respective Lagrange nodes situated on $\partial\Omega$. (Since we are concerned with linear elements on a simplicial mesh, $\mathcal{N}$ is equal to the set of internal nodes in the mesh.) To prove the result, it is sufficient to prove that each Lagrange nodal basis $\phi_{\nu}$ can be constructed as a recovery of an element-wise constant function $\psi_{\nu}\in V_h^0$, i.e., $\phi_{\nu}=\mathcal{E}(\psi_{\nu})$. 

To construct $\psi_{\nu}\in V_h^0$, we split $\mathcal{N}$ into a union of disjoint subsets as follows. Let $\mathcal{N}_1\subset \mathcal{N}$ denote the set of nodes $\nu_1$ for which there exists a triangle $T\in\mathcal{T}$ having $\nu_1$ as a node and the two remaining nodes in $\mathcal{N}_0$. For $i=2,\dots, r$, for some $r$, let $\mathcal{N}_i\subset \mathcal{N}\backslash \cup_{j=1}^{i-1}\mathcal{N}_j$ denote the set of nodes $\nu_i$ for which there exists a triangle $T\in\mathcal{T}$ having $\nu_i$ as a node and the two remaining nodes in $\mathcal{N}_{i-1}$. The existence of an $r\in\mathbb{N}$ such that $\mathcal{N}=\cup_{j=1}^{r}\mathcal{N}_j$ follows from the assumption that $\mathcal{T}$ is a regular simplicial mesh.

Having constructed a pairwise disjoint subdivision $\{\mathcal{N}_j\}_{j=1}^r$ of $\mathcal{N}$, we can now give an algorithm for the construction of $\psi_\nu$. Let $r_{\nu}\in\mathbb{N}$ be such that $\nu\in\mathcal{N}_{r_{\nu}}$ and set $\psi_{\nu}=1=(-1)^0$ on the element $T_{r_{\nu}}\in\mathcal{T}$ with vertex $\nu$ and the two other vertices $\nu_1,\nu_2$ in $\mathcal{N}_{r_{\nu}-1}$. For each of the $\nu_i,$ there exists a triangle  $T_{r_{\nu}-1}^i\in\mathcal{T}$, $i=1,2$, with one vertex being $\nu_i$ and the other two vertices being in $\mathcal{N}_{r_{\nu}-2}$. We now have 2 possibilities: 1) if $T_{r_{\nu}-1}^1$ and $T_{r_{\nu}-1}^2$ have a common vertex, say $\nu_3$, then there exists a triangle $T_{r_{\nu}-1}\in \mathcal{T}$ having vertices $\nu_1,\nu_2$ and $\nu_3$ on which we set $\psi_\nu=-1$ (notice that $T_{r_{\nu}-1}$ is necessarily different to the $T_{r_{\nu}-1}^i$'s); if $T_{r_{\nu}-1}^1$ and $T_{r_{\nu}-1}^2$ have no common vertex, we set $\psi_\nu=-1=(-1)^1$ on both $T_{r_{\nu}-1}^1,T_{r_{\nu}-1}^2$. We continue the above algorithm by setting $\psi_\nu=(-1)^j$ on $T_{r_{\nu}-j}^i$ for an index set $\mathcal{I}_j\ni i$, until we reach the boundary nodes in $\mathcal{N}_0$ for which the recovery imposes the homogeneous boundary conditions strongly. Noticing that on each node $\nu_i$ there exist exactly two elements with non-zero values one $+1$ and the other $-1$, we conclude that $\phi_{\nu}=c\mathcal{E}(\psi_{\nu})$ for some $c>0$, which already proves the result.
\end{proof}

Thus, we have proved the following a priori error bound.
\begin{corollary}
For $d=2$, let $\mathcal{T}$ be a regular and shape-regular simplicial mesh. Further, assuming that $u|_T\in H^2(T)$, $T\in\mathcal{T}$, and that $\alpha=1$ in Assumption \ref{stabilisation_bound}, we have that \eqref{apriori_reqs} holds with $q=1$.
\end{corollary}

So, we conclude that it is possible to have an optimally convergent finite element method for second order elliptic problems based on element-wise constant finite element spaces. This remarkable \emph{order increasing} property of the method when recovering from element-wise constants to conforming linear elements appears to hold for greater values of $r$. In particular, we have observed numerically that \emph{R-FEM with $r=s-1$, for $s=1,2,3$ converges with order $s$ in the $H^1$-norm and with order $s+1$ in the $L^2$-norm when the stabilisation coefficient $\sigma$ in \eqref{stabilisation} is chosen as
\begin{equation}\label{gen_stab}
\sigma = c_\sigma \mathcal{A} \mbf{h}^s,
\end{equation}
for $c_{\sigma}>0$ constant, i.e., $\alpha = s$ in \eqref{KP_stab}.} We refer to the numerical experiments in Section \ref{sec:test1} for a numerical illustration. The proof of this interesting property will be discussed elsewhere. We conjecture that this property is due to dimensional considerations since $\dim V_h^r > \dim \big(V_h^{r+1}\cap H^1_0(\Omega)\big)$ for $r= 1,2,3$. Nonetheless, it is not clear at this point why the scaling in \eqref{gen_stab} appears to be necessary for observing optimal convergence.

Interestingly, such order-increasing behaviour is not observed for recoveries of the type $r=s-2$, $s=2,3,\dots$. Nonetheless, as we shall see in the numerical experiments below, it can be still beneficial for the \emph{constant} of the convergence rate to recover lower order discontinuous spaces into higher order conforming ones for `compatible' choices of $\alpha$.

\section{An a posteriori error bound}\label{sec:apost}
 
To highlight further the flexibility offered by the proposed R-FEM framework, we also prove a basic residual-type a posteriori error bound for R-FEM with conforming recoveries. 
\begin{theorem}
  \label{the:apost}
  Let $u$ be the solution of (\ref{pde}) and $u_h, \mathcal E(u_h)$ be the R-FEM solution defined through (\ref{rem}) with $\mathcal{E}:V_h^r\to V_h^s\cap H^1_0(\Omega)$, $r\in \mathbb{N}_{0}$ and $s\in\mathbb{N}$. Assume also that the recovery operator $\mathcal{E}$ satisfies  \eqref{KP_stab} and that the stabilisation $s_h$ can be decomposed into local contributions $s_{h,T}$, so that
  \[
  s_h(w,v) = \sum_{T\in\mathcal{T}}s_{h,T}(w,v), \quad \text{for all } w,v\in v_h^r,
  \] (cf. Remark \ref{remark_local}). Then, we have the bound 
  \begin{equation}
    \norm{\sqrt{A}\nabla(u-\mathcal{E}(u_h))}{\Omega}^2 
    \leq  
   C  \sum_{T\in\mathcal{T}}\big(\eta_T^2+h_{T}^{2\alpha}s_{h,T}(u_h,u_h)+\eta_{A,T}^2\big) ,
  \end{equation}
  where
  \begin{equation}
    \eta_T
:    =
\Big(    \norm{\mbf{h} \left(f + \nabla\cdot \Pi A \nabla \mathcal E(u_h)\right)}{T}^2
    +
    \frac{1}{2} \norm{\mbf{h}^{1/2}\jump{\Pi A \nabla \mathcal E(u_h)}}{\partial T\backslash \partial\Omega}^2\Big)^{1/2},
  \end{equation}
and 
$
\eta_{A,T}:=\norm{(A-\Pi A)\nabla \mathcal{E}(u_h)}{T}
$, for $C>0$ an estimable-from-above constant depending on $C_0, \widetilde{C}_0$ and on the shape-regularity of the mesh; $\Pi\equiv \Pi_t:L^2(\Omega)\to V_h^t$ denotes the orthogonal $L^2$-projection onto $V_h^t$ for some $t\in \mathbb{N}\cup \{0\}$; when $\Pi$ is applied to tensors it will be understood as acting component-wise. 
\end{theorem}

\begin{proof}
We have, respectively,
\begin{equation*}
  \begin{split}
    \norm{\sqrt{A}\nabla(u-\mathcal{E}(u_h))}{\Omega}^2 
    &=    
    a(u-\mathcal{E}(u_h),u- \mathcal{E}(u_h) -\mathcal{E}(\chi))+s_h(u_h,\chi)
    \\ 
    &=\int_\Omega f v \ud x
    -
    a(\mathcal{E}(u_h),v)+s_h(u_h,\chi),
  \end{split}
\end{equation*}
from \eqref{Go} for $\chi\in V_h^r$ to be chosen precisely below, setting $v:=u- \mathcal{E}(u_h) -\mathcal{E}(\chi)$ for brevity.
Integration by parts and re-ordering of the terms yields
\begin{equation*}
  \begin{split}
    \norm{\sqrt{A}\nabla(u-\mathcal{E}(u_h))}{\Omega}^2 
    &=   
    \sum_{T\in\mathcal{T}}
    \int_T
    \left(f + \nabla\cdot \Pi A\nabla\mathcal{E}(u_h)\right) v \,\ud x
  -
    \int_{\Gamma_{\dint}} \jump{\Pi A\nabla\mathcal{E}(u_h)}v\, \ud s+s_h(u_h,\chi)\\
    &\quad +
    \int_\Omega
    (A- \Pi A)\nabla\mathcal{E}(u_h) \cdot\nabla v \,\ud x.
  \end{split}
\end{equation*}
Application of Cauchy-Schwarz inequality results in
\begin{equation}\label{first_bound_apost}
  \begin{split}
    \norm{\sqrt{A}\nabla(u-\mathcal{E}(u_h))}{\Omega}^2 
    &\leq  
   \Big( \sum_{T\in\mathcal{T}}\eta_T^2\Big)^{\frac 1 2}\big( \norm{\mbf{h}^{-1} v}{\Omega}^2
+  \norm{\mbf{h}^{-1/2}v}{\Gamma_{\dint}}^2\big)^{\frac 1 2}+s_h(u_h,\chi)\\
&\quad +
   \norm{(A- \Pi A)\nabla\mathcal{E}(u_h)}{\Omega} \norm{\nabla v }{\Omega}.
  \end{split}
\end{equation}
Selecting $\chi = \Pi_r(u - \mathcal{E}(u_h))$, (with $\Pi_r: L^2(\Omega)\to V_h^r$ the orthogonal $L^2$-projection operator onto $V_h^r$,) we have
\begin{equation*}
\begin{aligned}
 \norm{\mbf{h}^{-1} v}{\Omega} &\le \norm{\mbf{h}^{-1} (u-\mathcal{E}(u_h)-\chi)}{\Omega}
 +\norm{\mbf{h}^{-1} (\chi-\mathcal{E}(\chi))}{\Omega}\\
 &\le C\norm{\nabla(u-\mathcal{E}(u_h))}{\Omega}
 +\norm{\mbf{h}^{-1} (\chi-\mathcal{E}(\chi))}{\Omega}.
 \end{aligned}
\end{equation*}
Now, we have the bound
\begin{equation*}\label{rec_stab}
\begin{split}
  \norm{\mbf{h}^{-1} (\chi-\mathcal{E}(\chi))}{\Omega}^2
  &\le C_{KP,0}\norm{\mbf{h}^{-1/2}\jump{\chi}}{\Gamma}^2
  =C_{KP,0}\norm{\mbf{h}^{-1/2}\jump{u-\mathcal{E}(u_h)-\chi}}{\Gamma}^2\\
  &\le C \big(\norm{\mbf{h}^{-1}(u-\mathcal{E}(u_h)-\chi)}{\Omega}^2 +\sum_{T\in\mathcal{T}}\norm{\nabla (u-\mathcal{E}(u_h)-\chi)}{T}^2\big)
  \\
  &\le C\norm{\nabla(u-\mathcal{E}(u_h))}{\Omega}^2,
\end{split}
\end{equation*}
from the best approximation properties of the $L^2$-projection and its stability in the local $H^1$-seminorm. Completely analogously, we can show also that
\begin{equation*}
\begin{aligned}
 \norm{\mbf{h}^{-1/2} v}{\Gamma_{\dint}} &\le \norm{\mbf{h}^{-1/2} (u-\mathcal{E}(u_h)-\chi)}{\Gamma_{\dint}}
 +\norm{\mbf{h}^{-1/2} (\chi-\mathcal{E}(\chi))}{\Gamma_{\dint}}\\
 &\le C\norm{\nabla(u-\mathcal{E}(u_h))}{\Omega}
 +\norm{\mbf{h}^{-1/2} (\chi-\mathcal{E}(\chi))}{\Gamma_{\dint}}.
 \end{aligned}
\end{equation*}
and the bound
\begin{equation*}\label{rec_stab2}
\norm{\mbf{h}^{-1/2} (\chi-\mathcal{E}(\chi))}{\Gamma_{\dint}}^2\le C\norm{\nabla(u-\mathcal{E}(u_h))}{\Omega}^2.
\end{equation*}
Next, using a similar line of argument, we also have
\[
\norm{\nabla v }{\Omega} \le C\norm{\nabla(u-\mathcal{E}(u_h))}{\Omega}^2.
\]
Moreover, Assumptions \ref{stabilisation_ass2}, \ref{stabilisation_bound} and Remark \ref{remark_local} imply
\begin{equation}
\begin{aligned}
s_h(u_h,\chi)&=\sum_{T\in\mathcal{T}}s_{h,T}(u_h,\chi)
\le \widetilde{C}_0 \sum_{T\in\mathcal{T}}\big(h_T^{2\alpha}s_{h,T}(u_h,u_h)\big)^{1/2}\big(h_T^{-2\alpha}s_{h,T}(\chi,\chi)\big)^{1/2}\\
&\le \widetilde{C}_0C_0^{1/2} \Big(\sum_{T\in\mathcal{T}} h_T^{2\alpha}s_{h,T}(u_h,u_h)\Big)^{1/2}\Big(\sum_{T\in\mathcal{T}}h_T^{-2\alpha}\norm{\mbf{h}^{\alpha-1/2}\jump{\chi}}{\partial T}^2 \Big)^{1/2},
\end{aligned}
\end{equation}
and we also have
\[
\begin{aligned}
\sum_{T\in\mathcal{T}}h_T^{-2\alpha}\norm{\mbf{h}^{\alpha-1/2}\jump{\chi}}{\partial T}^2&\le C\sum_{T\in\mathcal{T}}\norm{\mbf{h}^{-1/2}\jump{\chi}}{\partial T}^2=  C\sum_{T\in\mathcal{T}}\norm{\mbf{h}^{-1/2}\jump{u-\mathcal{E}(u_h)-\chi}}{\partial T}^2\\
&\le C\norm{\nabla(u-\mathcal{E}(u_h))}{\Omega}^2.
\end{aligned}
\]

Combining the above estimates and applying them to estimate the respective terms on the right-hand side of \eqref{first_bound_apost}, we arrive at the a posteriori bound.
\end{proof}

Trivially, Assumption \ref{stabilisation_bound} implies also the bound
 \begin{equation}
    \norm{\sqrt{A}\nabla(u-\mathcal{E}(u_h))}{\Omega}^2 
    \leq  
   C  \sum_{T\in\mathcal{T}}\big(\eta_T^2+\eta_{A,T}^2\big)+C\norm{\mbf{h}^{2\alpha -1/2}\jump{u_h}}{\Gamma}^2.
  \end{equation}
 
 \begin{remark} 
The estimators $\eta_{A,T}$ can be thought as data oscillation \cite{pryer}. Indeed, setting $v_h=u_h$ in \eqref{rem}, using the Poincar\'e-Friedrichs inequality $\norm{\mathcal{E}(u_h)}{\Omega}\le C \norm{\nabla\mathcal{E}(u_h)}{\Omega}$ on the right-hand side of \eqref{rem}, along with the uniform ellipticity of $A$, we can arrive at the stability estimate
\[
\norm{\sqrt{A}\nabla \mathcal{E}(u_h)}{\Omega}\le C \norm{f}{\Omega},
\]
which, in turn, implies
\[
\sum_{T\in\mathcal{T}}\eta_{A,T}^2 \le C \norm{A-\Pi A}{L^{\infty}(\Omega)}^2\norm{f}{\Omega}^2.
\]
\end{remark}

\begin{remark}
We remark that the hypothesis that $\mathcal{E}$ should satisfy \eqref{KP_stab} of the previous theorem does \emph{not} imply that the result is applicable only when the averaging operator $\mathcal{E}_s$ is used in the R-FEM formulation. It merely states that for any recovery for which \eqref{KP_stab} holds, the above a posteriori bound also holds.
\end{remark}

Using standard tools, a lower bound for the error $\norm{\sqrt{A}\nabla(u-\mathcal{E}(u_h))}{\Omega}^2$ can also be given.
\begin{theorem}
With the hypotheses of Theorem \ref{the:apost}, for every $T\in\mathcal{T}$, we have the bound
\[
   \eta_T^2 \le C \norm{\sqrt{A}\nabla(u-\mathcal{E}(u_h))}{\omega_T}^2+\eta_{A,T}^2,   
\]
where $\omega_T:=\cup_{T'\in\mathcal{T}:\partial T'\cap\partial T\neq \emptyset} T'$.
\end{theorem}

\section{Implementation and Complexity}\label{sec:compl}

We shall now describe the implementation of the R-FEM and some of the properties of the resultant algebraic system. 

Let $N_C := \text{dim}\qp{V_h^s \cap H^1_{0}(\Omega)}$ and $N_D := \text{dim}\qp{V_h^r}$, the dimensions of the conforming and discontinuous spaces, respectively. The operator $\mathcal E : V_h^r \to V_h^s\cap H^1_{0}(\Omega)$ can be represented algebraically as a matrix $\vec E \in \reals^{N_C \times N_D}$ such that for $\vec u \in\reals^{N_D}$, $\vec E \vec u \in \reals^{N_C}$. Denote also by $\vec K_{\rm FEM} \in \reals^{N_C\times N_C}$ the stiffness matrix associated to the classical conforming finite element method \eqref{gabriel}. 
Then the algebraic form of \eqref{rem} is given by seeking $\vec u \in \reals^{N_D}$ such that
\begin{equation}
  \mathfrak A \vec u := \qp{\vec E^\transpose \vec K_{\rm FEM} \vec E + \vec S}\vec u =  \mathfrak{b}:=\vec E b ,
\end{equation}
 for $b\in \reals^{N_D}$ given by $b_i=\int_{\Omega}f\psi_i\ud x$, $i=1,\dots, N_D$ and $\psi_i\in V_h^r$ being a basis of $V_h^r$ and $\vec S$ being the algebraic representation of the stabilisation $s_h$.
For comparison's sake, we also introduce the dG stiffness matrix $\vec K_{IP} \in \reals^{N_D\times N_D}$ associated to the interior penalty dG method with solution $u_h^{IP}\in V_h^r$.

\begin{theorem}[Condition number estimate]\label{cond_num_est}
For $d=2$, let $\mathcal T$ be a locally quasiuniform triangulation, and assume that the basis functions of $V_h^r$ used in the computation have local support. Assume also that the recovery operator $\mathcal{E}$ satisfies  \eqref{KP_stab} and the stabilisation is such that Assumptions \ref{stabilisation_bound} and \ref{stabilisation_ass2} hold for $0\le \alpha\le 1$. Then, the $\ell_2$-condition number $ \kappa(\mathfrak A)$ of the R-FEM stiffness matrix $\mathfrak A$ satisfies
  \begin{equation}
    \kappa(\mathfrak A) \leq C (N_D+|\log(\underline{\mbf{h}}^2N_D)|),
  \end{equation}
with $\underline{\mbf{h}}:=\min_{\Omega} \mbf{h}$,  for $C$ a positive constant independent of the dimension of $\mathfrak{A}$. Also, under the same assumptions, for $d=3$, we have $
    \kappa(\mathfrak A) \leq C N_D.$
\end{theorem}

\begin{proof} Without loss of generality, we assume $\mbf{h}\le 1$. It is sufficient to find  $\lambda,\Lambda>0$ such that $\lambda\vec v^\transpose \vec v \le
\vec v^\transpose \mathfrak{A}\vec v \le \Lambda \vec v^\transpose \vec v $ for all $\vec v\in \reals^{N_D}$; then  $\kappa(\mathfrak A)\le \Lambda/\lambda$. 
From \eqref{KP_stab} and Assumption \ref{stabilisation_bound}, we have, respectively,
\[
\begin{split}
\vec v^\transpose \mathfrak{A}\vec v & = a(\mathcal{E}(v_h),\mathcal{E}(v_h))+s_h(v_h,v_h) 
\le  \norm{\sqrt{A}\nabla\mathcal{E}(v_h)}{\Omega}^2
     + C_0\norm{\mbf{h}^{\alpha-1/2}\jump{v_h}}{\Gamma}^2 \\
&     \le C\Big(
      \sum_{T\in\mathcal{T}}\norm{\nabla v_h}{T}^2
     + \norm{\mbf{h}^{\min\{\alpha,0\}-1/2}\jump{v_h}}{\Gamma}^2\Big)\\
&     \le C \sum_{T\in\mathcal{T}} \big( 1	     
      +h_T^{2\min\{\alpha,0\}}\big)	\norm{v_h}{L^\infty(T)}^2
		\le C  \underline{\mbf{h}}^{2\min\{\alpha,0\}}   \vec v^\transpose \vec v\le C   \vec v^\transpose \vec v,
     \end{split}
\]
since the basis functions used have local support.

Now, working analogously as in the proof of Theorem 4.1 from \cite{bs}, for $p\in(2,\infty]$, we have
\[
\begin{split}
  \vec v^\transpose \vec v &\le  C \sum_{T\in\mathcal{T}} \norm{v_h}{L^{\infty}(T)}^2
  \le C \sum_{T\in\mathcal{T}} h_T^{-4/p}\norm{v_h}{L^p(T)}^2
  \le C \Big(\sum_{T\in\mathcal{T}} h_T^{-4/(p-2)}\Big)^{1-2/p}\norm{v_h}{L^p(\Omega)}^2\\
&\le C \Big(\sum_{T\in\mathcal{T}} h_T^{-4/(p-2)}\Big)^{1-2/p}\big(\norm{\mathcal{E}(v_h)}{L^p(\Omega)}^2
+\norm{\mbf{h}^{(p-2)/2p}(v_h-\mathcal{E}(v_h))}{\Omega}^2\big)
\\
&\le C \Big(\sum_{T\in\mathcal{T}} h_T^{-4/(p-2)}\Big)^{1-2/p}\big(p\norm{\nabla \mathcal{E}(v_h)}{\Omega}^2
+\norm{\mbf{h}^{1-1/p}\jump{v_h}}{\Gamma}^2\big)
\\
&\le C (p+\bar{\mbf{h}}^{3/2-1/p-\alpha}) \Big(\sum_{T\in\mathcal{T}} h_T^{-4/(p-2)}\Big)^{1-2/p}\big(\norm{\sqrt{A}\nabla \mathcal{E}(v_h)}{\Omega}^2
+c_0\norm{\mbf{h}^{\alpha-1/2}\jump{v_h}}{\Gamma}^2\big)
\\
&\le C p  (\underline{\mbf{h}}^2N_D)^{-2/p} N_D\vec v^\transpose \mathfrak{A}\vec v,
 \end{split}
\]
using standard inverse estimates, H\"older's inequality, \eqref{KP_stab}, Sobolev Imbedding Theorem, and setting $\bar{\mbf{h}}:=\max_{\Omega}\mbf{h}$, respectively. Selecting now $p=\max\{2,|\log(\underline{\mbf{h}}^2N_D)|\}$, the result follows.

The result for $d=3$ is completely analogous, but it is omitted for brevity, as it essentially follows the steps from \cite{bs}.
\end{proof}

 The above result shows that R-FEM enjoys similar conditioning to standard FEM and dG methods for $0\le \alpha\le 1$. 

For the case $r=s-1$, for $s=1,2,3$ discussed above, an inspection of the proof of Theorem \ref{cond_num_est} shows that the conditioning of $\frak A$ deteriorates when $\alpha>1$. This is an interesting challenge for R-FEM implementation, which will be investigated further elsewhere. Crucially, however, $\mathfrak A$ is the respective stiffness matrix of the linear system solving for $u_h$ and \emph{not} of the recovered solution $\mathcal{E}(u_h)$, for which the respective linear system is expected to be better conditioned in practice. This could be a starting point for the development of preconditioning strategies for R-FEM.

\subsection{Matrix structure}

Due to the compact support of the finite element basis functions, linear systems resulting from the finite element discretisation of PDEs are typically sparse. Sparsity of the stiffness matrix is traditionally an indicator of complexity of the solution of the respective linear system. To this end, we begin by studing the sparsity of the R-FEM stiffness matrix and compare with FEM and dG counterparts. 

The recovery procedure can be implemented through local \emph{transition matrices}. These are defined by locally rewriting the basis functions of the underlying nonconforming space in terms of the space used for reconstruction.

\begin{example}[$\cE : V_h^0 \to V_h^1\cap H^1_0(\Omega)$]
  Let $M_0 = \dim{\poly{0}(K)} = 1$ and $M_1 = \dim{\poly{1}(K)} = 3$. The local transistion matrix is a mapping which takes the form $\vec T_K \in \reals^{M_0 \times M_1}$ where
  \begin{equation*}
    \vec T_K =
    \begin{bmatrix}
      1, 1, 1
    \end{bmatrix}.
  \end{equation*}
  The global recovery matrix $\vec E$ can then be constructed by appropriately weighting and summing over all elements.
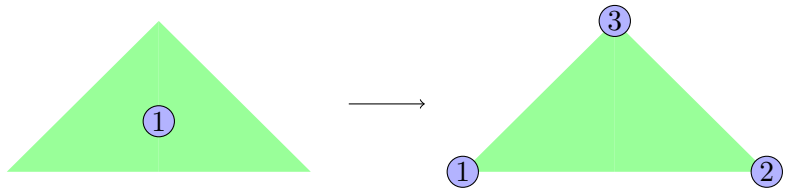
\begin{figure}[h!]
\caption{The degrees of freedom of the $\poly{0}$ dG space and that of the reconstructed space $\poly{1}$.}
\label{fig:p0p1dofs}
  \begin{center}
    \begin{tikzpicture}      
      
      % draw the left subtriangle in original place
      \path[fill=green!40] 
      (0,0) node (N0) {}
      -- (2,0) node (N5) {}
      -- (2,2) node (N2) {}
      -- cycle;
      
      % a nodez for labeling the coarse triangle
      \path[coordinate]
      (2,.67) node (K) {};
      
      % draw the left subtriangle in shifted place
      \path[fill=green!40,xshift=6cm] 
      (0,0) node (NL1) {}
      -- (2,0) node (NL2) {}
      -- (2,2) node (NL0) {}
      -- cycle;
      
      % two nodez for labeling the left subtriangle
      \foreach \s in {0,1}{
	\path[coordinate,xshift=6cm*\s]
	(1.33,.67) node (L\s) {};
      }
      
      % draw the right subtriangle in original place
      \path[fill=green!40]
      (4,0) node (N1) {}
      -- (2,2)
      -- (2,0)
      -- cycle
      ;
      
      % draw the right subtriangle in shifted place
      \path[fill=green!40,xshift=6cm]
      (4,0) node (NR0) {}
      -- (2,2) node (NR1) {}
      -- (2,0) node (NR2) {}
      -- cycle
      ;
      
      % two nodez for labeling the right subtriangle
      \foreach \s in {0,1}{
	\path[coordinate,xshift=6cm*\s]
	(2.67,.67) node (R\s) {};
      }

      \node (KL0) at (NL0) [circle,inner sep=1pt,fill=blue!30,draw] {3};
      \node (KL0) at (NL1) [circle,inner sep=1pt,fill=blue!30,draw] {1};
      \node (KR0) at (NR0) [circle,inner sep=1pt,fill=blue!30,draw] {2};
      
      % label original triangle
      \node at (K) [circle,inner sep=1pt,fill=blue!30,draw] { 1 }; %center
      \node at (L1) [rectangle,inner sep=1pt,xshift=.6cm,yshift=-.1cm ] {};
      
      \draw[->] (4.5,0.9) -- (5.5,0.9);
      %node[pos=0.5,above,yshift=.2cm] {$\cE$};
    \end{tikzpicture}
\end{center}
\vspace{-.8cm}
\end{figure}  
\end{example}

\begin{example}[$\cE : V_h^1 \to V_h^2\cap H^1_0(\Omega)$]
For $M_0 = \dim{\poly{1}(K)} = 3$ and $M_1 = \dim{\poly{2}(K)} = 6$. The local transition matrix $\vec T_K \in \reals^{M_0 \times M_1}$ is given by
  \begin{equation*}
    \vec T_K =
    \begin{bmatrix}
      1 & 0 & 0 & 0 & 1/2 & 1/2
      \\
      0 & 1 & 0 & 1/2 & 0 & 1/2
      \\
      0 & 0 & 1 & 1/2 & 1/2 & 0
    \end{bmatrix}.
  \end{equation*}
 % The global recovery matrix $\vec E$ can then be constructed by appropriately summing over all elements.
\begin{figure}[]
\caption{The degrees of freedom of the $\poly{1}$ dG space and that of the reconstructed space $\poly{2}$.}
\label{fig:p1p2dofs}
  \begin{center}
    \begin{tikzpicture}      
      % draw the left subtriangle in original place
      \path[fill=green!40] 
      (0,0) node (N0) {}
      -- (2,0) node (N5) {}
      -- (2,2) node (N2) {}
      -- cycle;
      
      % a nodez for labeling the coarse triangle
      \path[coordinate]
      (2,.67) node (K) {};
      
      % draw the left subtriangle in shifted place
      \path[fill=green!40,xshift=6cm] 
      (0,0) node (NL1) {}
      -- (2,0) node (NL2) {}
      -- (2,2) node (NL0) {}
      -- cycle;
      
      % two nodez for labeling the left subtriangle
      \foreach \s in {0,1}{
	\path[coordinate,xshift=6cm*\s]
	(1.33,.67) node (L\s) {};
      }
      
      % draw the right subtriangle in original place
      \path[fill=green!40]
      (4,0) node (N1) {}
      -- (2,2)
      -- (2,0)
      -- cycle
      ;
      
      % draw the right subtriangle in shifted place
      \path[fill=green!40,xshift=6cm]
      (4,0) node (NR0) {}
      -- (2,2) node (NR1) {}
      -- (2,0) node (NR2) {}
      -- cycle
      ;
      
      % two nodez for labeling the right subtriangle
      \foreach \s in {0,1}{
	\path[coordinate,xshift=6cm*\s]
	(2.67,.67) node (R\s) {};
      }
      
      % define the non-nodal DOF for coarse triangle
      \path[coordinate]
      (3,1) node (N3) {}
      (1,1) node (N4) {}
      ;
      
      % define the non-nodal DOF for shifted left subtriangle
      \path[coordinate,xshift=6cm]
      (1,0) node (NL3) {}
      (2,1) node (NL4) {}
      (1,1) node (NL5) {}
      ;

      % define the non-nodal DOF for shifted right subtriangle
      \path[coordinate,xshift=6cm]
      (2,1) node (NR3) {}
      (3,0) node (NR4) {}
      (3,1) node (NR5) {}
      ;      
      
      % draw the circles for the dofs on coarse triangle
      \node (K0) at (N0) [circle,inner sep=1pt,fill=blue!30,draw] {1};
      \node (K1) at (N1) [circle,inner sep=1pt,fill=blue!30,draw] {2};
      \node (K2) at (N2) [circle,inner sep=1pt,fill=blue!30,draw] {3};
      
      % draw the circles for the old dofs on subtriangles
      \foreach \s in {0,1,2}{
	\node (KL0) at (NL0) [circle,inner sep=1pt,fill=blue!30,draw] {3};
	\node (KR0) at (NR0) [circle,inner sep=1pt,fill=blue!30,draw] {2};
        \node (KL1) at (NL1) [circle,inner sep=1pt,fill=blue!30,draw] {1};

        \node (KL5) at (NL5) [circle,inner sep=1pt,fill=blue!30,draw] {5};
        \node (KR5) at (NR5) [circle,inner sep=1pt,fill=blue!30,draw] {4};
        \node (KR4) at (NR2) [circle,inner sep=1pt,fill=blue!30,draw] {6};
	
      }
      
      \draw[->] (4.5,0.9) -- (5.5,0.9);
    \end{tikzpicture}
\end{center}
\vspace{-.8cm}
\end{figure}
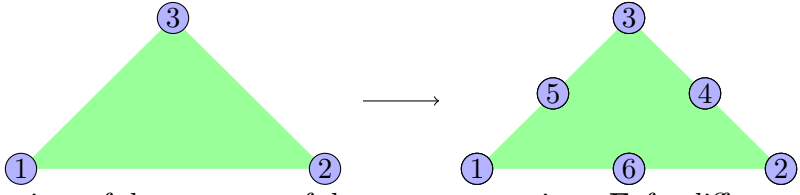  
\end{example}

When considering the resulting \emph{full} linear systems, we observe that the R-FEM yield system matrices that have the same dimension but are less sparse than the respective dG counterparts on the same space $\tilde{V}$. Indeed, the R-FEM stiffness matrix appears to have the sparsity pattern of a (conforming or non-conforming) finite element method applied to $V_h^s$ but the dimension of $V_h^r$, which may in general be larger than the one of $\tilde{V}$ on standard simplicial meshes. 
Moreover, the higher the regularity of the recovered space $\tilde{V}$, the larger the bandwidth of the resultant matrix. We refer to Figure \ref{fig:recovered-matrix} for a visualisation of the sparsity structure of the recovery matrix $\vec E$ and to Figure \ref{fig:full-matrix} for a comparison of the sparsity patterns for $\mathfrak{A}$, $K_{\rm FEM}$ and $K_{\rm IP}$ over the same mesh.

\begin{figure}[]
  \caption[]
  {\label{fig:recovered-matrix}
    %%%%%%%%%%%%%%%%%%%%%%%%%%%%%%%%%%%%%%%%%%%%%%%%%%%%%%%%%%%%%%%%%
    Illustrations of the structure of the recovery matrices, $\vec E$, for different recovery degrees over the same mesh. The mesh is of criss-cross type and is formed of $256$ elements.
    %%%%%%%%%%%%%%%%%%%%%%%%%%%%%%%%%%%%%%%%%%%%%%%%%%%%%%%%%%%%%%%%%% 
  }
  \begin{center}
        %%%%%%%%%%%%%%%%%%%%%%%%%%%%%%%%%%%%%%%%%%%%%%%%%%%%%%%%%%%%%%%%%%%% 
    \subfloat[{
        %%%%%%%%%%%%%%%%%%%%%%%%%%%%%%%%%%%%%%%%%%%%%%%%%%%%%%%%%%%%%%%
        {
               $\vec {E}\in\reals^{256\times 145}$ for $\mathcal{E}:V_h^0\to V_h^1\cap H^1_0(\Omega)$.
        }
      %%%%%%%%%%%%%%%%%%%%%%%%%%%%%%%%%%%%%%%%%%%%%%%%%%%%%%%%%%%%%%%% 
    }]{
      \includegraphics[scale=\figscale,width=0.3\figwidth]{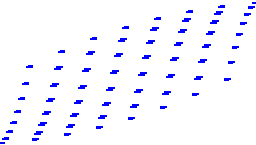}
    }    %%%%%%%%%%%%%%%%%%%%%%%%%%%%%%%%%%%%%%%%%%%%%%%%%%%%%%%%%%%%%%%%%%%%
    \hspace{.2cm}
    %%%%%%%%%%%%%%%%%%%%%%%%%%%%%%%%%%%%%%%%%%%%%%%%%%%%%%%%%%%%%%%%%%%% 
    \subfloat[{
      %%%%%%%%%%%%%%%%%%%%%%%%%%%%%%%%%%%%%%%%%%%%%%%%%%%%%%%%%%%%%%% 
        {
          $\vec {E}\in\reals^{768\times 145}$ for $\mathcal{E}:V_h^1\to V_h^1\cap H^1_0(\Omega)$. 
        }
        %%%%%%%%%%%%%%%%%%%%%%%%%%%%%%%%%%%%%%%%%%%%%%%%%%%%%%%%%%%%%%%% 
    }]{
      \includegraphics[scale=\figscale,width=0.31\figwidth]{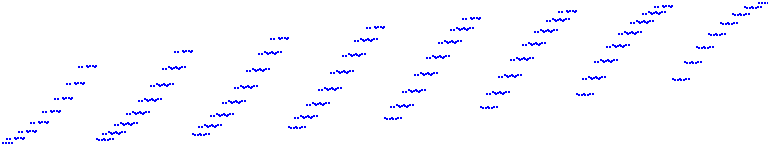}
    }
    %%%%%%%%%%%%%%%%%%%%%%%%%%%%%%%%%%%%%%%%%%%%%%%%%%%%%%%%%%%%%%%%%%%% 
    %%%%%%%%%%%%%%%%%%%%%%%%%%%%%%%%%%%%%%%%%%%%%%%%%%%%%%%%%%%%%%%%%%%% 
\hspace{.2cm}    
    \subfloat[{
        %%%%%%%%%%%%%%%%%%%%%%%%%%%%%%%%%%%%%%%%%%%%%%%%%%%%%%%%%%%%%%%
        {
        $\vec E \in\reals^{768\times 545}$  for $\mathcal{E}:V_h^1\to V_h^2\cap H^1_0(\Omega)$. 
        }
      %%%%%%%%%%%%%%%%%%%%%%%%%%%%%%%%%%%%%%%%%%%%%%%%%%%%%%%%%%%%%%%% 
    }]{
      \includegraphics[scale=\figscale,width=0.3\figwidth]{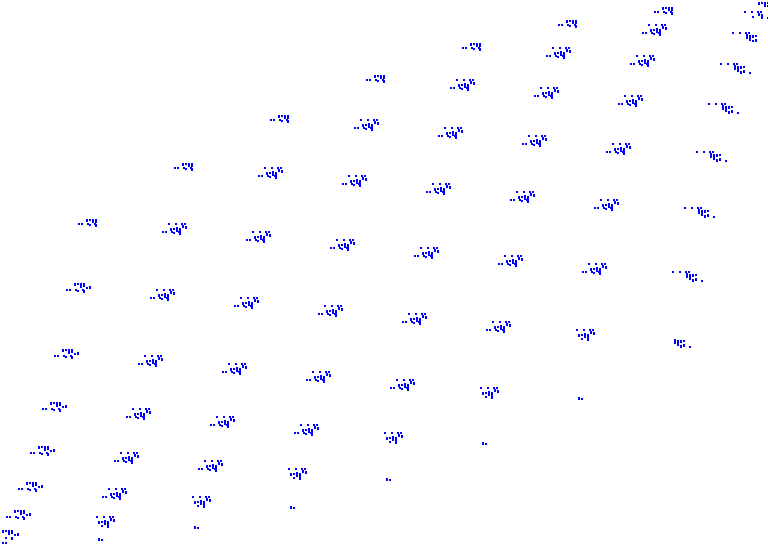}
    }
    %%%%%%%%%%%%%%%%%%%%%%%%%%%%%%%%%%%%%%%%%%%%%%%%%%%%%%%%%%%%%%%%%%%% 
  \end{center}
  \vspace{-1cm}
\end{figure}
\begin{figure}
  \caption[]
  {\label{fig:full-matrix}
    %%%%%%%%%%%%%%%%%%%%%%%%%%%%%%%%%%%%%%%%%%%%%%%%%%%%%%%%%%%%%%%%%
    Illustrations of the structure of the full R-FEM system matrix, $\mathfrak A$, for the R-FEM with $\mathcal{E}: V_h^0\to V_h^1\cap H^1_0(\Omega)$, and comparison to $\vec K_{\rm IP}$ and $\vec K_{\rm FEM}$ with discontinuous and continuous linear elements, respectively, over the same $256$-element criss-cross type mesh.
    %%%%%%%%%%%%%%%%%%%%%%%%%%%%%%%%%%%%%%%%%%%%%%%%%%%%%%%%%%%%%%%%%% 
  }
  \begin{center}
        %%%%%%%%%%%%%%%%%%%%%%%%%%%%%%%%%%%%%%%%%%%%%%%%%%%%%%%%%%%%%%%%%%%% 
    \subfloat[{
        %%%%%%%%%%%%%%%%%%%%%%%%%%%%%%%%%%%%%%%%%%%%%%%%%%%%%%%%%%%%%%%
        {
          $\mathfrak A\in\reals^{256\times 256}$
        }
      %%%%%%%%%%%%%%%%%%%%%%%%%%%%%%%%%%%%%%%%%%%%%%%%%%%%%%%%%%%%%%%% 
    }]{
      \includegraphics[scale=\figscale,width=0.3\figwidth]{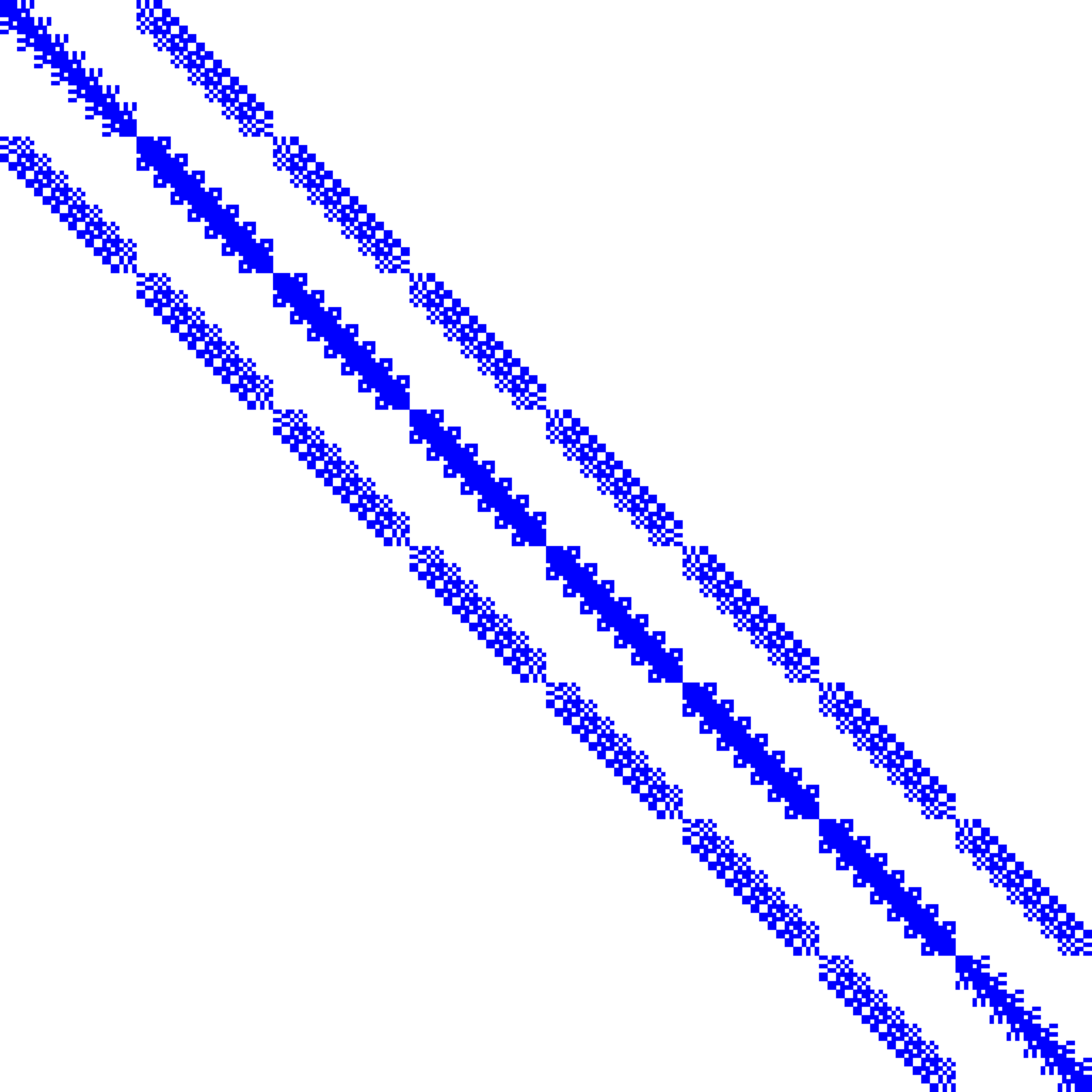}
    }    %%%%%%%%%%%%%%%%%%%%%%%%%%%%%%%%%%%%%%%%%%%%%%%%%%%%%%%%%%%%%%%%%%%%
    \hfill
    %%%%%%%%%%%%%%%%%%%%%%%%%%%%%%%%%%%%%%%%%%%%%%%%%%%%%%%%%%%%%%%%%%%% 
    \subfloat[{
        %%%%%%%%%%%%%%%%%%%%%%%%%%%%%%%%%%%%%%%%%%%%%%%%%%%%%%%%%%%%%%%
        {
         $\vec K_{\rm FEM} \in\reals^{145\times 145}$.
        }
      %%%%%%%%%%%%%%%%%%%%%%%%%%%%%%%%%%%%%%%%%%%%%%%%%%%%%%%%%%%%%%%% 
    }]{
      \includegraphics[scale=\figscale,width=0.3\figwidth]{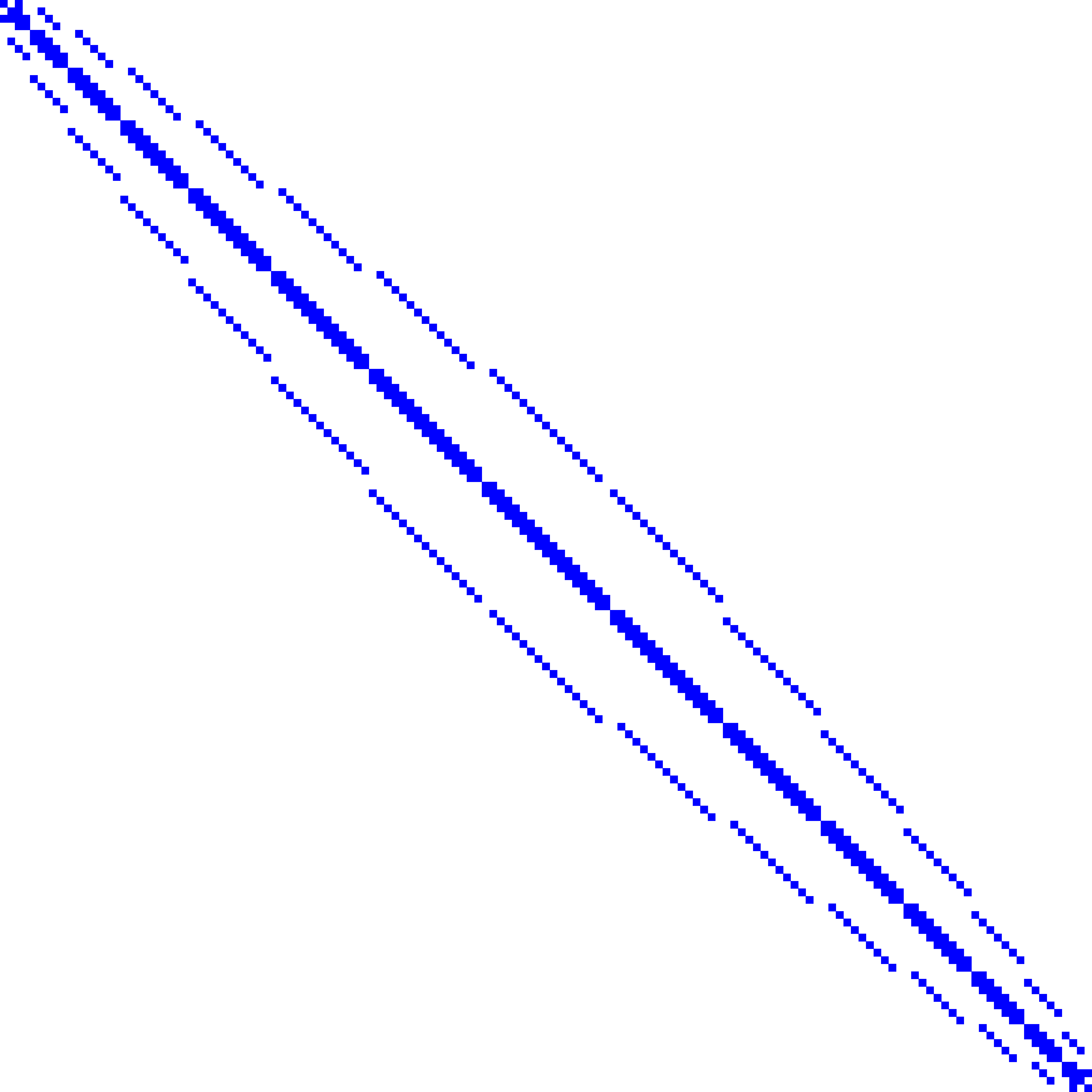}
    }
    \hfill
    \subfloat[{
        %%%%%%%%%%%%%%%%%%%%%%%%%%%%%%%%%%%%%%%%%%%%%%%%%%%%%%%%%%%%%%%%
        {$\vec K_{\rm IP} \in\reals^{768\times 768}$.
        }
        %%%%%%%%%%%%%%%%%%%%%%%%%%%%%%%%%%%%%%%%%%%%%%%%%%%%%%%%%%%%%%%% 
    }]{
      \includegraphics[scale=\figscale,width=0.3\figwidth]{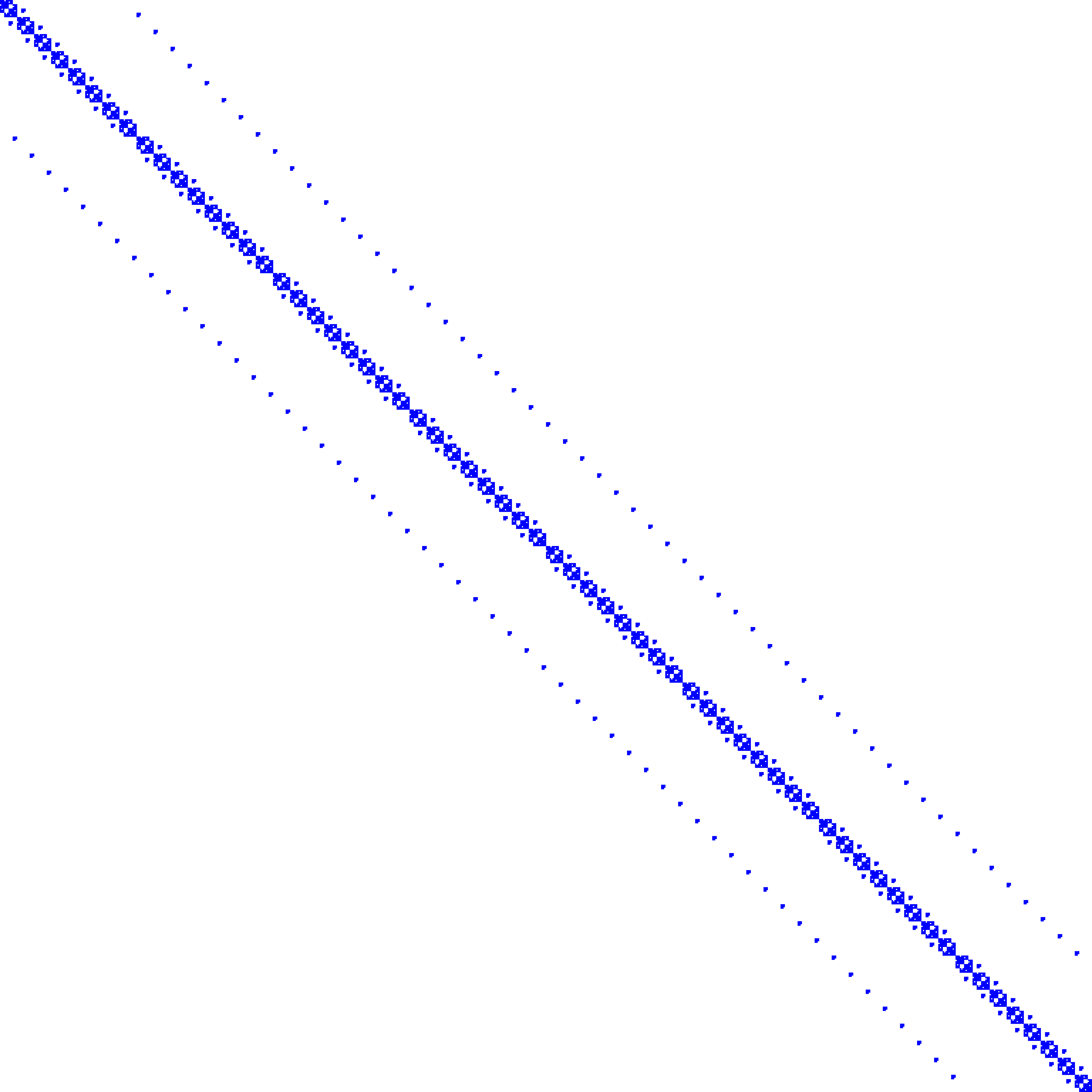}
    }
    %%%%%%%%%%%%%%%%%%%%%%%%%%%%%%%%%%%%%%%%%%%%%%%%%%%%%%%%%%%%%%%%%%%% 
  \end{center}
  \vspace{-1cm}
\end{figure}

\section{Numerical Benchmarking}
\label{sec:numerics}

We now illustrate the performance of the R-FEM through a series of numerical experiments. The implementation of R-FEM used is available at \cite{REMcode}.

\subsection{Test 1 -- Asymptotic behaviour approximating a smooth solution}\label{sec:test1}

As a first test, we consider $A=I$ and the domain $\Omega = [0,1]^2$. We fix $f$ such that the exact solution is given by
\begin{equation}
  \label{eq:smooth-sol}
  u(x,y) = \sin(\pi x) \sin(\pi y),
\end{equation}
and approximate $\Omega$ through a uniformly generated, criss-cross triangular type mesh to test the asymptotic behaviour of the numerical approximation. The results are summarised in Figure \ref{fig:convergence-smooth} (a) -- (c), and confirm the theoretical findings in Section \ref{sec:apriori}. More specifically, upon selecting $\sigma$ as in \eqref{gen_stab}, we witness $s=r+1$ order of convergence in the energy norm error $u-\mathcal{E}(u_h)$, for $r=0,1,2$ and respective $s+1=r+2$ order of convergence in the $L^2$-norm. 

As a further test, we consider the case $r=1$, $s=3$ and $\alpha=3$, i.e., when recovering discontinuous linear elements into conforming cubic elements. In this case, only first order convergence rate in the $H^1$-norm and second order convergence in the $L^2$-norm are observed, i.e., there is \emph{no} improvement in the convergence \emph{rate}. However, this recovery can be still beneficial for the \emph{constant} of the convergence rate: in Figure \ref{fig:convergence-smooth}(d) we observe a reduction in the absolute energy norm error against numerical degrees of freedom along with a comparison to the (classical) conforming linear FEM.
\begin{figure}[h!]
  \caption[]
  {\label{fig:convergence-smooth}
    %%%%%%%%%%%%%%%%%%%%%%%%%%%%%%%%%%%%%%%%%%%%%%%%%%%%%%%%%%%%%%%%%
    Convergence plots for the R-FEM (\ref{rem}) for Test 1. We measure error norms involving the R-FEM solution, $u_h$, and its reconstruction, $\mathcal E(u_h)$. 
    %%%%%%%%%%%%%%%%%%%%%%%%%%%%%%%%%%%%%%%%%%%%%%%%%%%%%%%%%%%%%%%%%% 
  }
  \begin{center}
        %%%%%%%%%%%%%%%%%%%%%%%%%%%%%%%%%%%%%%%%%%%%%%%%%%%%%%%%%%%%%%%%%%%% 
    \subfloat[{
        %%%%%%%%%%%%%%%%%%%%%%%%%%%%%%%%%%%%%%%%%%%%%%%%%%%%%%%%%%%%%%%
        {
          R-FEM with $\mathcal E : V_h^0 \to V_h^1\cap H^1_0(\Omega)$
        }
      %%%%%%%%%%%%%%%%%%%%%%%%%%%%%%%%%%%%%%%%%%%%%%%%%%%%%%%%%%%%%%%% 
    }]{
      \includegraphics[scale=\figscale,width=0.52\figwidth]{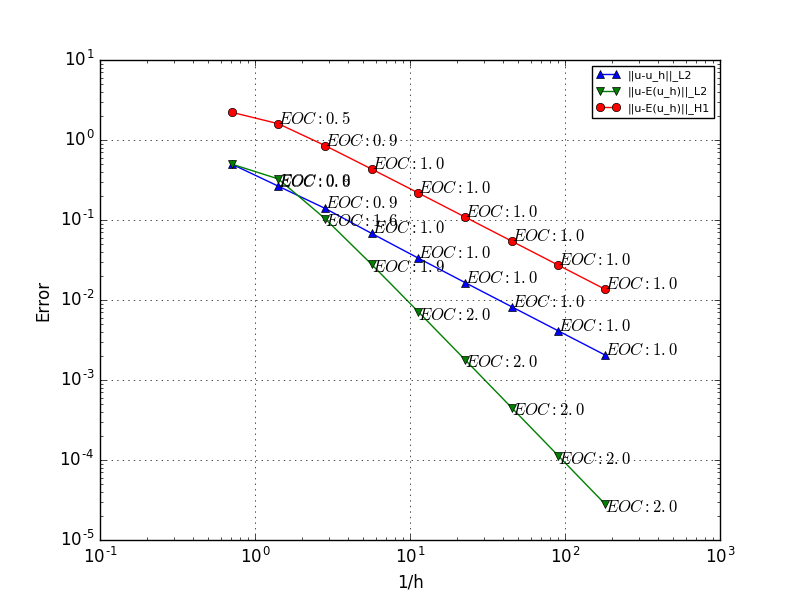}
    }    %%%%%%%%%%%%%%%%%%%%%%%%%%%%%%%%%%%%%%%%%%%%%%%%%%%%%%%%%%%%%%%%%%%
 %   \hfill
    %%%%%%%%%%%%%%%%%%%%%%%%%%%%%%%%%%%%%%%%%%%%%%%%%%%%%%%%%%%%%%%%%%%% 
    \subfloat[{
      %%%%%%%%%%%%%%%%%%%%%%%%%%%%%%%%%%%%%%%%%%%%%%%%%%%%%%%%%%%%%%% 
        {
          R-FEM with $\mathcal E : V_h^1 \to V_h^2\cap H^1_0(\Omega)$
        }
        %%%%%%%%%%%%%%%%%%%%%%%%%%%%%%%%%%%%%%%%%%%%%%%%%%%%%%%%%%%%%%%% 
    }]{
      \includegraphics[scale=\figscale,width=0.52\figwidth]{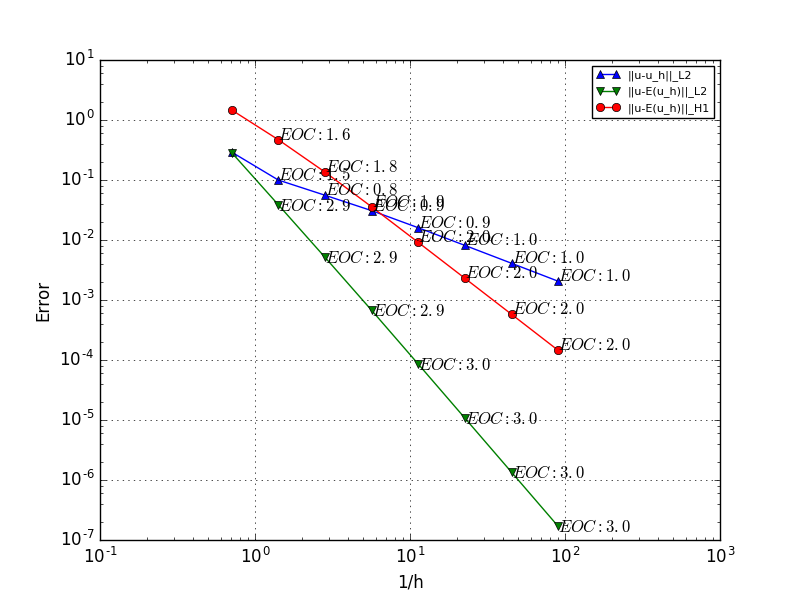}
    }\\
        \subfloat[{
      %%%%%%%%%%%%%%%%%%%%%%%%%%%%%%%%%%%%%%%%%%%%%%%%%%%%%%%%%%%%%%% 
        {
          R-FEM with $\mathcal E : V_h^2 \to V_h^3\cap H^1_0(\Omega)$
        }
        %%%%%%%%%%%%%%%%%%%%%%%%%%%%%%%%%%%%%%%%%%%%%%%%%%%%%%%%%%%%%%%% 
    }]{
      \includegraphics[scale=\figscale,width=0.52\figwidth]{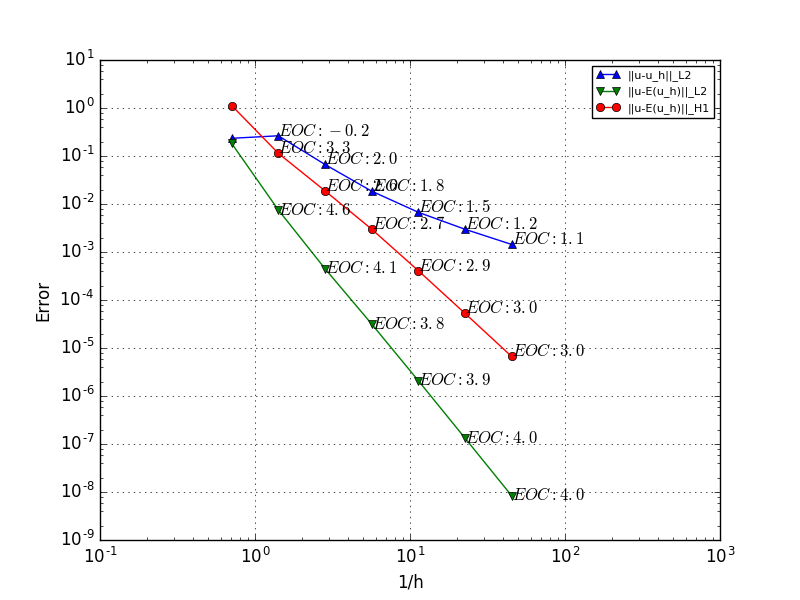}
    }
      \subfloat[{
      %%%%%%%%%%%%%%%%%%%%%%%%%%%%%%%%%%%%%%%%%%%%%%%%%%%%%%%%%%%%%%% 
        {
          R-FEM with $\mathcal E : V_h^1 \to V_h^3\cap H^1_0(\Omega)$ and linear FEM
        }
        %%%%%%%%%%%%%%%%%%%%%%%%%%%%%%%%%%%%%%%%%%%%%%%%%%%%%%%%%%%%%%%% 
    }]{
      \includegraphics[scale=\figscale,width=0.52\figwidth]{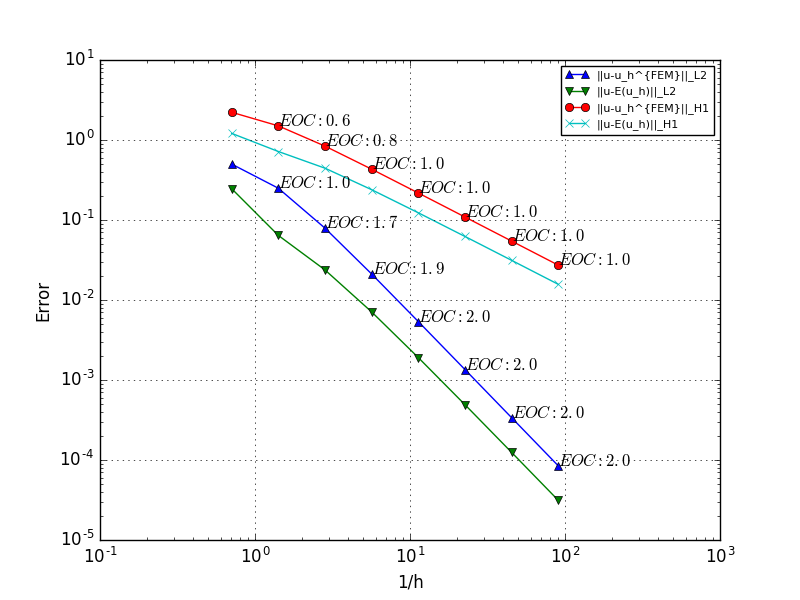}
    }
    %%%%%%%%%%%%%%%%%%%%%%%%%%%%%%%%%%%%%%%%%%%%%%%%%%%%%%%%%%%%%%%%%%%% 
  \end{center}
  \vspace{-1cm}
\end{figure}

Next, we compare the lowest order R-FEM method with $\mathcal E : V_h^0 \to V_h^1\cap H^1_0(\Omega)$ and \eqref{stabilisation} with \eqref{sigma_rneqs} against the dG interior penalty method for the same problem and perform a series of asymptotic benchmarks varying the penalty parameter $c_\sigma$. Recall that for R-FEM $c_\sigma > 0$ can be chosen arbitrarily, as opposed to interior penalty dG methods where $\sigma$ should be given by \eqref{sigma_reqs} with $c_\sigma$ sufficiently large so as to ensure coercivity of the bilinear form. The results are given in Figure \ref{fig:sigmatestdg}.  We fix the dG penalty parameter to be $\sigma_{\rm IP} = 10 \mbf{h}^{-1}$ and vary the R-FEM penalty parameter, $c_\sigma$. Note that R-FEM behaves comparably to the interior penalty method and there is a region of values of $c_\sigma$ where $\norm{u - \cE(u_h)}{L^2(\Omega)} < \norm{u - u_h^{IP}}{L^{2}(\Omega)}$. For small values of $c_\sigma$ we see $\norm{\nabla(u - \cE(u_h))}{\Omega} \approx \norm{u - u_h^{IP}}{\rm dG}$ with $\norm{\cdot}{\rm dG}$ denoting energy-like dG-norm defined as
\[
\norm{w}{\rm dG}:=\Big(\sum_{T\in\mathcal{T}}\norm{\nabla w}{T}^2+\norm{\sqrt{\sigma_{\rm IP}}\jump{w}}{\Gamma}\Big)^{1/2}.
\]
\begin{figure}
  \caption[]
  {\label{fig:sigmatestdg}
    %%%%%%%%%%%%%%%%%%%%%%%%%%%%%%%%%%%%%%%%%%%%%%%%%%%%%%%%%%%%%%%%%
    Convergence plots for lowest order R-FEM for (\ref{eq:smooth-sol}) and comparison with dG method for different values of $c_\sigma$. 
    %%%%%%%%%%%%%%%%%%%%%%%%%%%%%%%%%%%%%%%%%%%%%%%%%%%%%%%%%%%%%%%%%% 
  }
  \begin{center}
        %%%%%%%%%%%%%%%%%%%%%%%%%%%%%%%%%%%%%%%%%%%%%%%%%%%%%%%%%%%%%%%%%%%% 
    \subfloat[{
        %%%%%%%%%%%%%%%%%%%%%%%%%%%%%%%%%%%%%%%%%%%%%%%%%%%%%%%%%%%%%%%
        {
         Convergence for $\sigma = 10 \mbf{h}$.
        }
      %%%%%%%%%%%%%%%%%%%%%%%%%%%%%%%%%%%%%%%%%%%%%%%%%%%%%%%%%%%%%%%% 
    }]{
      \includegraphics[scale=\figscale,width=0.5\figwidth]{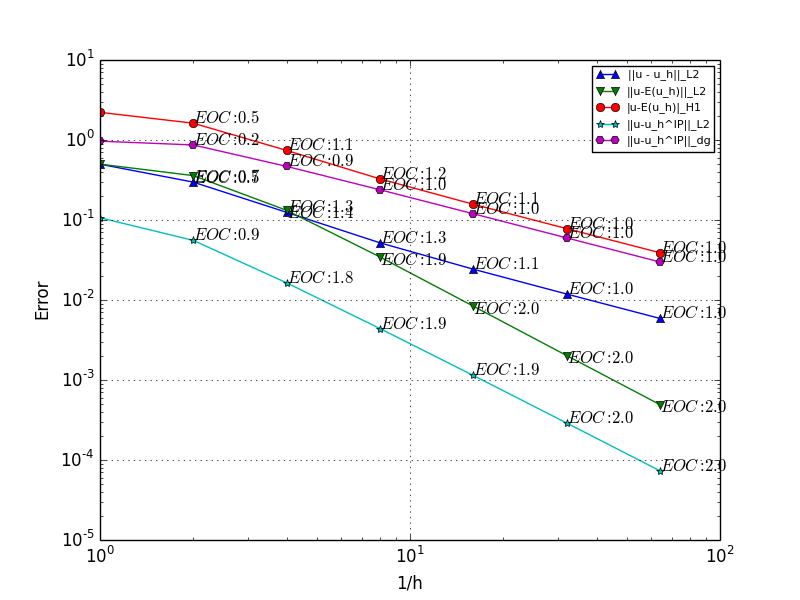}
    }    %%%%%%%%%%%%%%%%%%%%%%%%%%%%%%%%%%%%%%%%%%%%%%%%%%%%%%%%%%%%%%%%%%%
  %  \hfill
    %%%%%%%%%%%%%%%%%%%%%%%%%%%%%%%%%%%%%%%%%%%%%%%%%%%%%%%%%%%%%%%%%%%% 
    \subfloat[{
      %%%%%%%%%%%%%%%%%%%%%%%%%%%%%%%%%%%%%%%%%%%%%%%%%%%%%%%%%%%%%%% 
        {
          Convergence for $\sigma = \mbf{h}$.
        }
        %%%%%%%%%%%%%%%%%%%%%%%%%%%%%%%%%%%%%%%%%%%%%%%%%%%%%%%%%%%%%%%% 
    }]{
      \includegraphics[scale=\figscale,width=0.5\figwidth]{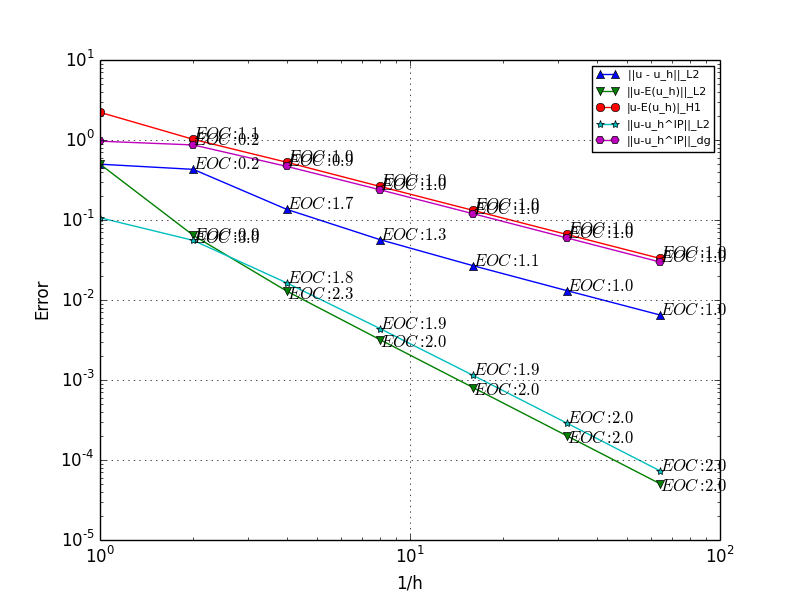}
    }
    %%%%%%%%%%%%%%%%%%%%%%%%%%%%%%%%%%%%%%%%%%%%%%%%%%%%%%%%%%%%%%%%%%%% 
    \hfill
    %%%%%%%%%%%%%%%%%%%%%%%%%%%%%%%%%%%%%%%%%%%%%%%%%%%%%%%%%%%%%%%%%%%% 
    \subfloat[{
      %%%%%%%%%%%%%%%%%%%%%%%%%%%%%%%%%%%%%%%%%%%%%%%%%%%%%%%%%%%%%%% 
        {
          Convergence for $\sigma = 0.1\mbf{h}$.
        }
        %%%%%%%%%%%%%%%%%%%%%%%%%%%%%%%%%%%%%%%%%%%%%%%%%%%%%%%%%%%%%%%% 
    }]{
      \includegraphics[scale=\figscale,width=0.5\figwidth]{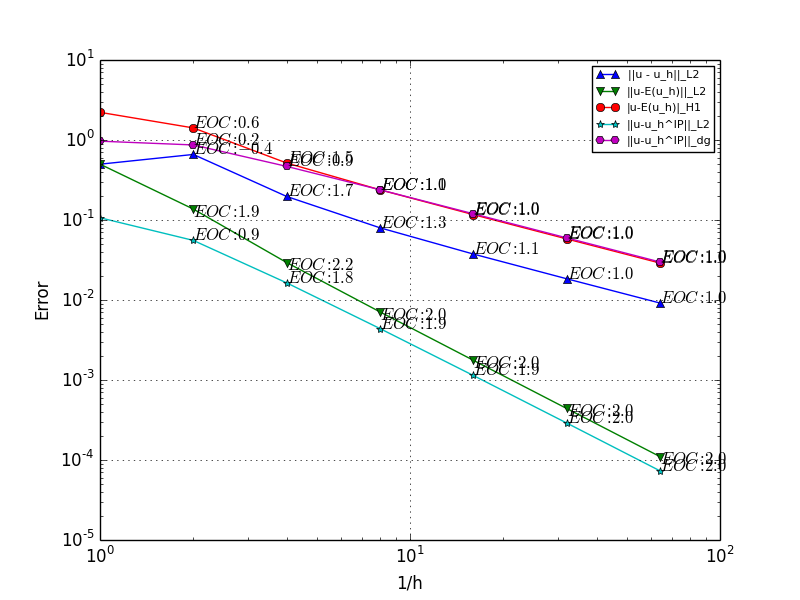}
    }
    %%%%%%%%%%%%%%%%%%%%%%%%%%%%%%%%%%%%%%%%%%%%%%%%%%%%%%%%%%%%%%%%%%%% 
   % \hfill
    %%%%%%%%%%%%%%%%%%%%%%%%%%%%%%%%%%%%%%%%%%%%%%%%%%%%%%%%%%%%%%%%%%%% 
    \subfloat[{
        %%%%%%%%%%%%%%%%%%%%%%%%%%%%%%%%%%%%%%%%%%%%%%%%%%%%%%%%%%%%%%% 
        {
         Convergence for $\sigma =0.01 \mbf{h}$.
        }
        %%%%%%%%%%%%%%%%%%%%%%%%%%%%%%%%%%%%%%%%%%%%%%%%%%%%%%%%%%%%%%%% 
    }]{
      \includegraphics[scale=\figscale,width=0.5\figwidth]{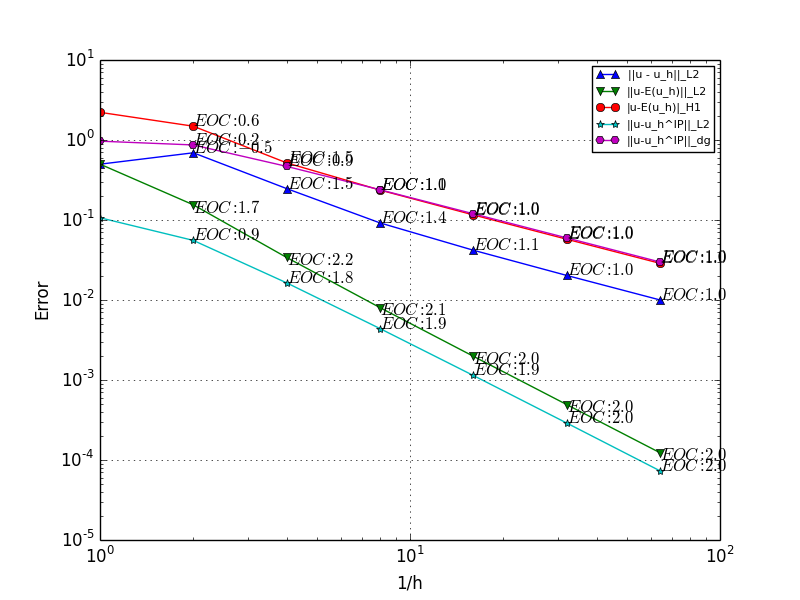}
    }
    %%%%%%%%%%%%%%%%%%%%%%%%%%%%%%%%%%%%%%%%%%%%%%%%%%%%%%%%%%%%%%%%%%%% 
  \end{center}
\end{figure}

\subsection{Test 2 -- Asymptotic behaviour approximating a singular solution}

We choose $f$ so that
\begin{equation}
  u(x,y) = r^{2/3} \sin\qp{2\theta/3} \qp{x^2 - 1} \qp{y^2 - 1},
\end{equation}
where $r^2 = x^2 + y^2$ and $\theta = \arctan(y/x)$ are the associated
polar coordinates, over the domain $\Omega=(-1,1)^2 \backslash
[0,1)\times (-1,0]$. We triangulate the domain with a criss-cross
triangular mesh. Note the solution is not $H^2(\Omega)$, as the radial
derivative is singular at the origin. We test the convergence of the
R-FEM approximation under a uniform $h$-refinement. We also take the
opportunity to test the a posteriori estimate given in Theorem
\ref{the:apost}. Results are summarised in Figure \ref{fig:lshape},
where convergence plots for $r=0,s=1$ are presented.
\begin{figure}[ht!]
  \caption[]
  {\label{fig:lshape}
    %%%%%%%%%%%%%%%%%%%%%%%%%%%%%%%%%%%%%%%%%%%%%%%%%%%%%%%%%%%%%%%%%
    Convergence plots for the R-FEM approximation of a solution with singular derivatives of the Poisson problem on the re-entrant corner problem. 
    %%%%%%%%%%%%%%%%%%%%%%%%%%%%%%%%%%%%%%%%%%%%%%%%%%%%%%%%%%%%%%%%%% 
  }
  \begin{center}
        %%%%%%%%%%%%%%%%%%%%%%%%%%%%%%%%%%%%%%%%%%%%%%%%%%%%%%%%%%%%%%%%%%%% 
    \subfloat[{
        %%%%%%%%%%%%%%%%%%%%%%%%%%%%%%%%%%%%%%%%%%%%%%%%%%%%%%%%%%%%%%%
        {
          R-FEM with $\mathcal E:\mathcal V_h^0\to V_h^1\cap H^1_0(\Omega)$.
        }
      %%%%%%%%%%%%%%%%%%%%%%%%%%%%%%%%%%%%%%%%%%%%%%%%%%%%%%%%%%%%%%%% 
    }]{
      \includegraphics[scale=.42]{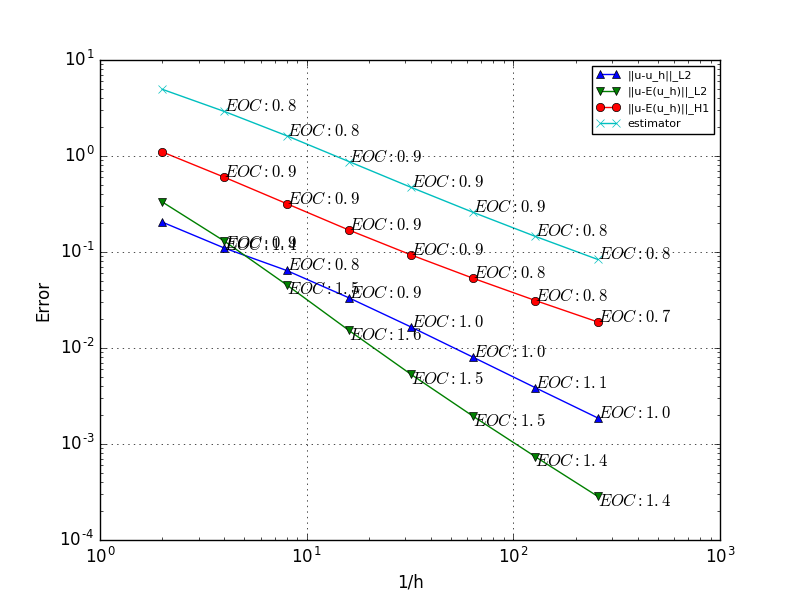}
    }    %%%%%%%%%%%%%%%%%%%%%%%%%%%%%%%%%%%%%%%%%%%%%%%%%%%%%%%%%%%%%%%%%%%
    \subfloat[{
      %%%%%%%%%%%%%%%%%%%%%%%%%%%%%%%%%%%%%%%%%%%%%%%%%%%%%%%%%%%%%%% 
        {
          Computed solution $u_h\in V_h^0$ from Fig.~\ref{fig:lshape}(a).
        }
        %%%%%%%%%%%%%%%%%%%%%%%%%%%%%%%%%%%%%%%%%%%%%%%%%%%%%%%%%%%%%%%% 
    }]{\hspace{-1cm}
      \includegraphics[scale=.14]{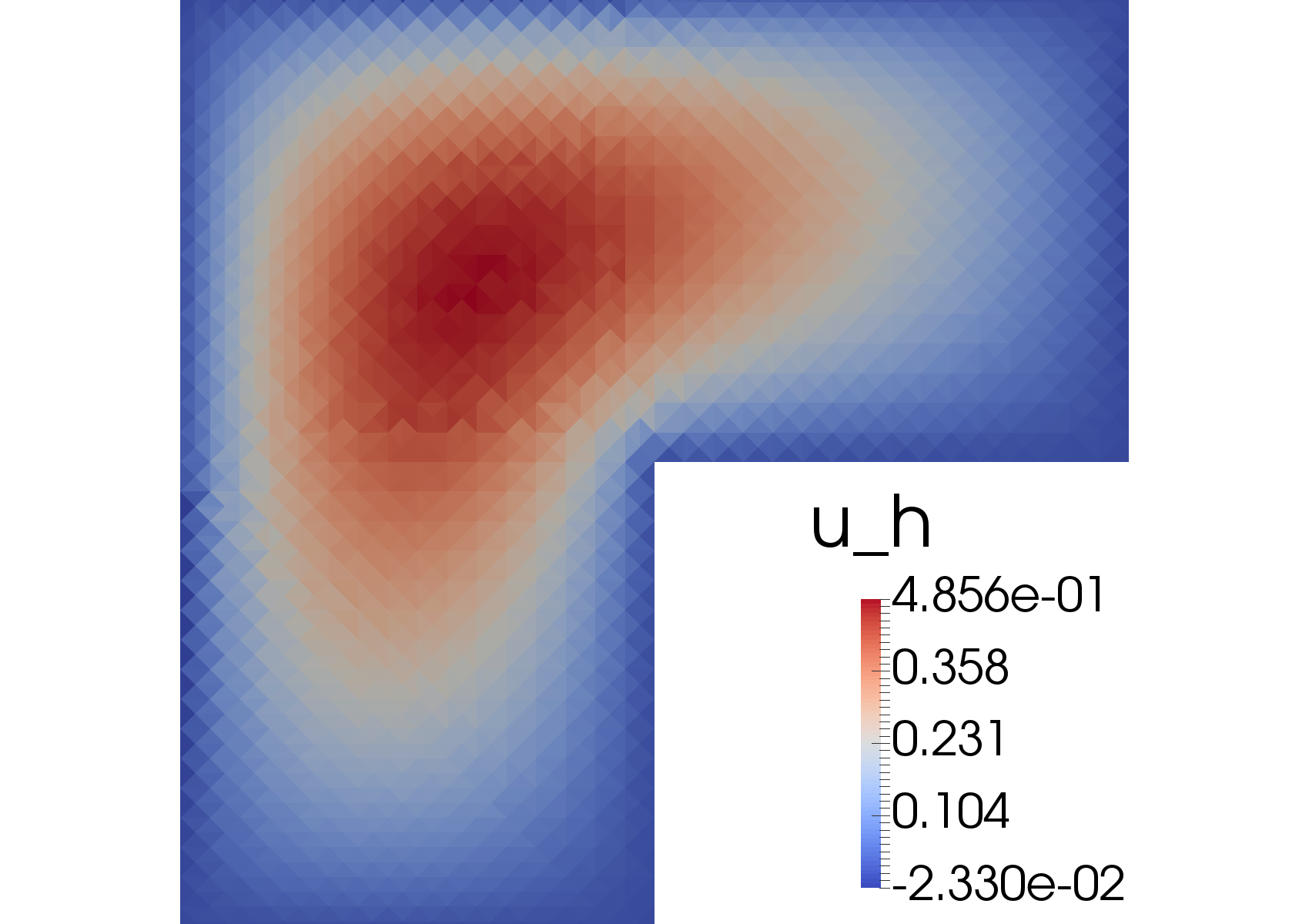}
    }   %%%%%%%%%%%%%%%%%%%%%%%%%%%%%%%%%%%%%%%%%%%%%%%%%%%%%%%%%%%%%%%%%%%

  \end{center}
  \vspace{-.7cm}
\end{figure}

\subsection{Test 3 -- Mesh adaptivity}

We continue with testing the behaviour of the estimator within a standard adaptive algorithm of the form
$
  \text{SOLVE} \to \text{ESTIMATE} \to \text{MARK} \to \text{REFINE}.
$
 In this context, over a coarse mesh we \emph{solve} the R-FEM problem, \emph{estimate} the error using the a posteriori error bound from Theorem \ref{the:apost}, \emph{mark} a subset of elements where the a posteriori indicator is large, and \emph{refine} those elements. The algorithm is then iterated until a prescribed tolerance is achieved. For the marking step we make use of the maximum strategy, that is we mark a set of elements $\mathcal M \subset \mathcal T_h$ such that
\begin{equation*}
  \sum_{T \in \mathcal M} \eta_T \leq \theta \max_{T \in T_h} \eta_T,
\end{equation*}
where $\theta \in (0,1]$ is the ratio of refined elements. In our experiments we choose $\theta = 0.25$. The numerical results are summarised in Figure \ref{fig:lshape-adapt}, where convergence plots for the lowest order ($r=0,s=1$) adaptive R-FEM approximation to the solution of the above problem are shown, indicating that the adaptive approximation is optimally convergent and its effectivity index, i.e., the ratio between the estimator and the exact error, is small and remains bounded.

\begin{figure}[ht!]
  \caption[]
  {\label{fig:lshape-adapt}
    %%%%%%%%%%%%%%%%%%%%%%%%%%%%%%%%%%%%%%%%%%%%%%%%%%%%%%%%%%%%%%%%%
    Convergence for the lowest order ($r=0,s=1$) adaptive R-FEM approximation, meshes produced and its effectivity.
    %%%%%%%%%%%%%%%%%%%%%%%%%%%%%%%%%%%%%%%%%%%%%%%%%%%%%%%%%%%%%%%%%% 
  }
  \begin{center}
        %%%%%%%%%%%%%%%%%%%%%%%%%%%%%%%%%%%%%%%%%%%%%%%%%%%%%%%%%%%%%%%%%%%% 
    \subfloat[{
        %%%%%%%%%%%%%%%%%%%%%%%%%%%%%%%%%%%%%%%%%%%%%%%%%%%%%%%%%%%%%%%
        {
          Convergence against number of degrees of freedom.
        }
      %%%%%%%%%%%%%%%%%%%%%%%%%%%%%%%%%%%%%%%%%%%%%%%%%%%%%%%%%%%%%%%% 
    }]{\hspace{-1cm}
      \includegraphics[scale=.45]{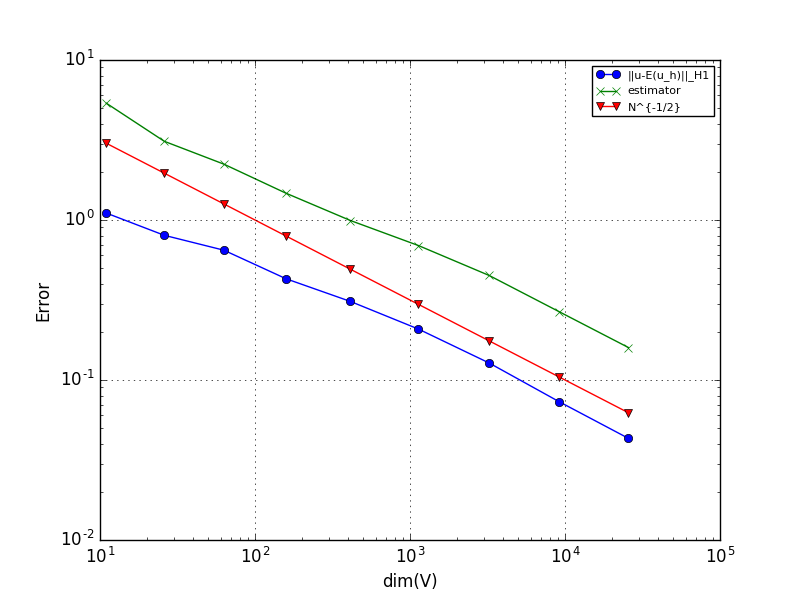}
    }    %%%%%%%%%%%%%%%%%%%%%%%%%%%%%%%%%%%%%%%%%%%%%%%%%%%%%%%%%%%%%%%%%%%
  %  \hfill
    \subfloat[{
        %%%%%%%%%%%%%%%%%%%%%%%%%%%%%%%%%%%%%%%%%%%%%%%%%%%%%%%%%%%%%%%
        {
          The effectivity index of the estimator.
        }
      %%%%%%%%%%%%%%%%%%%%%%%%%%%%%%%%%%%%%%%%%%%%%%%%%%%%%%%%%%%%%%%% 
    }]{\hspace{-1cm}
      \includegraphics[scale=.45]{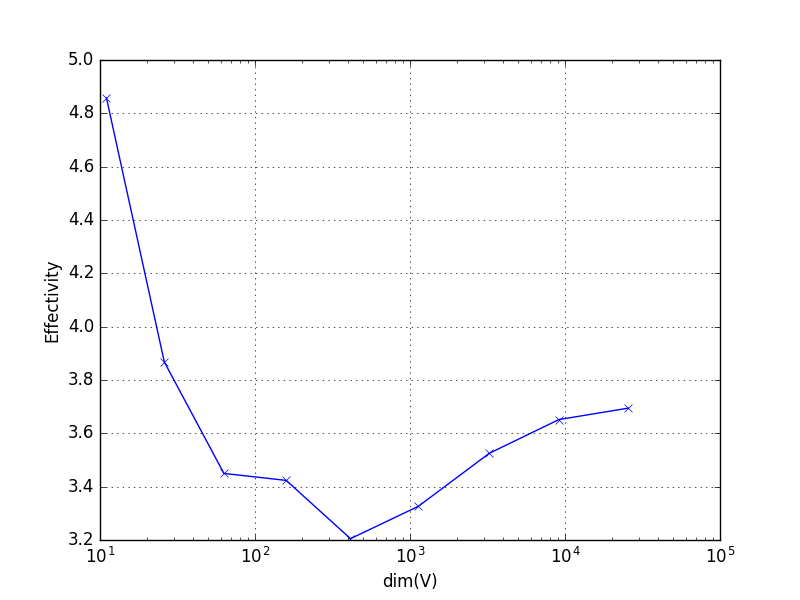}
    }    %%%%%%%%%%%%%%%%%%%%%%%%%%%%%%%%%%%%%%%%%%%%%%%%%%%%%%%%%%%%%%%%%%%
    \hfill
    %%%%%%%%%%%%%%%%%%%%%%%%%%%%%%%%%%%%%%%%%%%%%%%%%%%%%%%%%%%%%%%%%%%% 
    \subfloat[{
      %%%%%%%%%%%%%%%%%%%%%%%%%%%%%%%%%%%%%%%%%%%%%%%%%%%%%%%%%%%%%%% 
        {
          The initial mesh.
        }
        %%%%%%%%%%%%%%%%%%%%%%%%%%%%%%%%%%%%%%%%%%%%%%%%%%%%%%%%%%%%%%%% 
    }]{\hspace{0cm}
      \includegraphics[scale=.32]{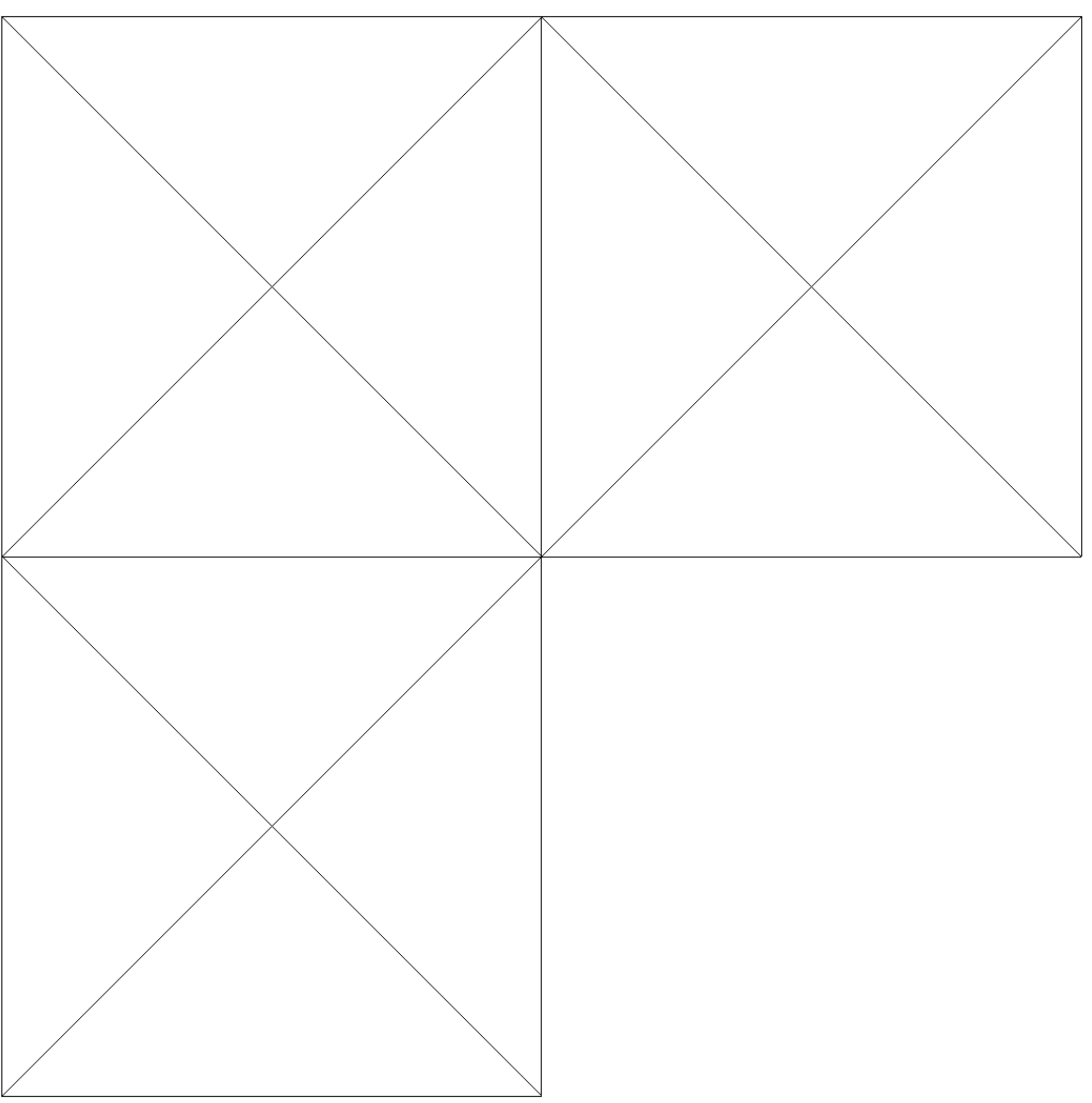}
    }   %%%%%%%%%%%%%%%%%%%%%%%%%%%%%%%%%%%%%%%%%%%%%%%%%%%%%%%%%%%%%%%%%%%
 %   \hfill
    %%%%%%%%%%%%%%%%%%%%%%%%%%%%%%%%%%%%%%%%%%%%%%%%%%%%%%%%%%%%%%%%%%%% 
    \subfloat[{
      %%%%%%%%%%%%%%%%%%%%%%%%%%%%%%%%%%%%%%%%%%%%%%%%%%%%%%%%%%%%%%% 
        {
          The mesh after 11 adaptive algorithm iterations.
        }
        %%%%%%%%%%%%%%%%%%%%%%%%%%%%%%%%%%%%%%%%%%%%%%%%%%%%%%%%%%%%%%%% 
    }]{\hspace{1cm}
      \includegraphics[scale=.33]{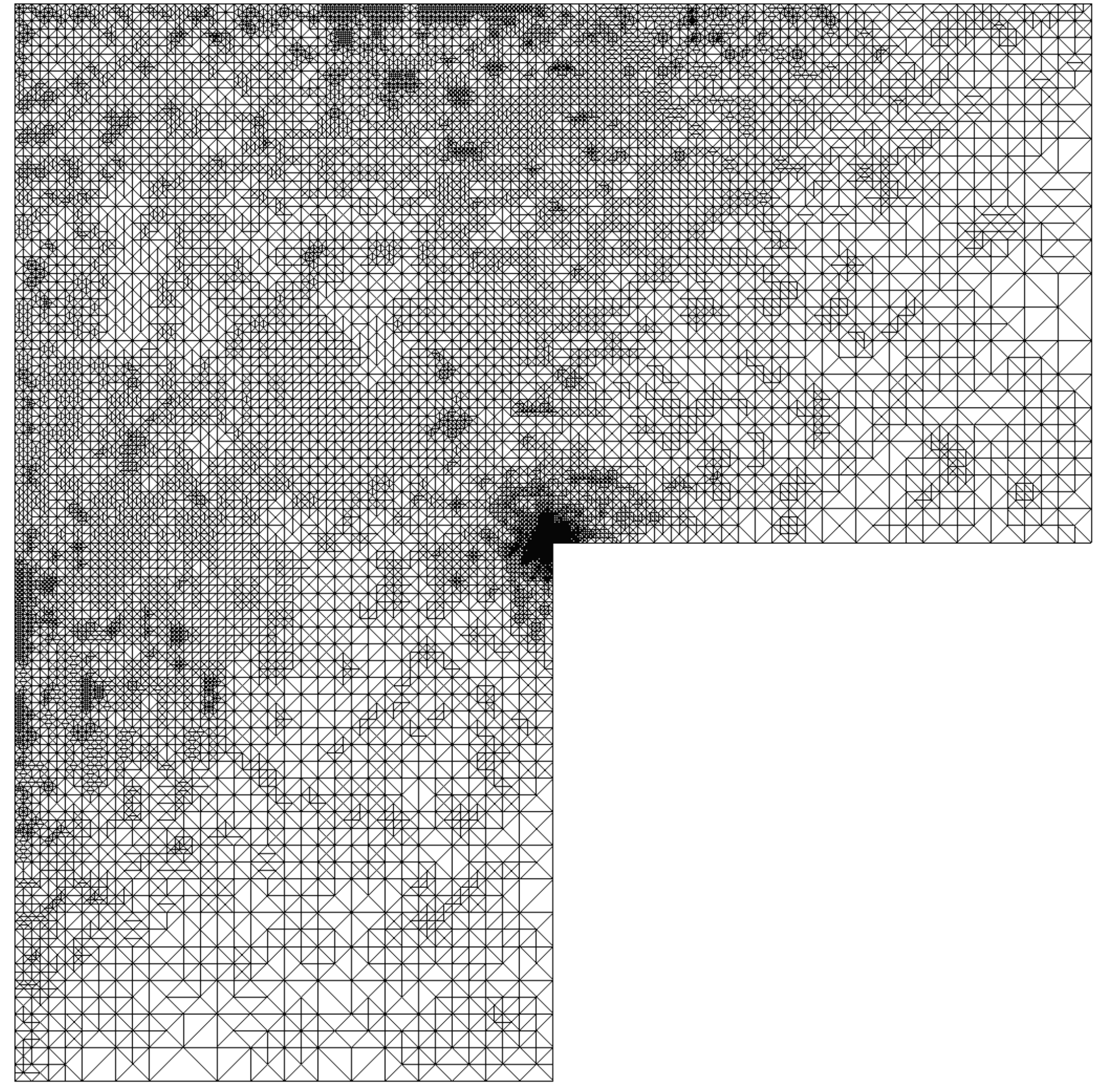}
    }
  \end{center}
\end{figure}

\newcommand{\E}{\mathcal E}

\subsection{Test 4 -- R-FEM for convection-diffusion problems}\label{sec:conv-diff}

To showcase the versatility of the R-FEM framework, we consider a convection-diffusion problem solved by a `hybrid' R-FEM discretisation for the second order term and an upwinded dG method for the first order term. We envisage such a combination of R-FEM with stabilised methods for stiff problems as a typical area of applications for R-FEM, aiming to combined the advantages of conforming (or classical non-conforming) discretisations with considerable freedom in the treatment of lower order terms. To this end, we consider the convection-diffusion problem
\begin{equation}\label{conv-diff}
  -\epsilon \Delta u + \vec w\cdot \nabla u  +cu = f,\qquad \text{on } \Omega,
\end{equation}
where $\mathbb{R}\ni \epsilon \ll 1$ represents viscosity and $\vec w\in C(\bar{\Omega})^d$ and $c\in L^2(\Omega)$ are some prescribed convection vector field and reaction coefficients, respectively. We supplement \eqref{conv-diff} with homogeneous Dirichlet (no-slip) boundary conditions on $\partial \Omega$.

The flexibility in the design of R-FEM described above, which employs element-wise discontinuous discrete spaces, allows naturally the treatment of the diffusion term in a conforming manner and, \emph{simultaneously}, the convection term in an upwinded nonconforming fashion. This allows for the construction of stable methods in the convection-dominated regime. To this end, consider the R-FEM discretisation: find $u_h \in V^0_h$ (with $\E(u_h)\in V^1_h \cap H^1_{0}(\Omega)$), such that:
\begin{equation}
  \label{eq:rem-upwind}
  \begin{aligned}
  &\int_\Omega \epsilon \nabla \E(u_h) \cdot \nabla \E(v_h) +c u_h v_h \ud x +\int_\Gamma \sigma \jump{u_h} \cdot \jump{v_h}\ud s \\
  +& \sum_{T\in\mathcal{T}}\Big(\int_T \vec w\cdot\nabla u_h v_h \ud x  + \int_{\partial_- T}( \vec w \cdot \vec n )\ujump{u_h} v_h^+\ud s \Big)
  =
  \int_\Omega f \E(v_h)\ud x \Foreach v_h \in V^0_h,
    \end{aligned}
\end{equation}
where $\ujump{v}(x):=\lim_{\delta\to 0^+}\big(v(x+\delta\vec w(x))-v(x-\delta\vec w(x))\big)$, for (almost every) $x\in \partial T$, is the \emph{upwind jump} and
\begin{equation}
  \partial_-T := \{ x\in\partial T : \vec w(x) \cdot \vec n (x) < 0 \},
\end{equation}
denotes the \emph{inflow boundary} of an element $T\in\mathcal{T}$; when $x\in \partial \Omega$, we set the respective downwind value to zero in the definition of $\ujump{\cdot}$. 

We shall investigate the convergence and stability properties of the proposed method \eqref{eq:rem-upwind} for $\Omega=(0,1)^2$ by setting
\begin{equation}
(a)\quad  \vec w = \big( (2x_2-1)(1-x_1^2), 2x_1 x_2 (x_2-1)\big)^{\transpose}
\quad \text{ and } \quad (b) \quad \vec w = (1,1)^{\transpose},
\end{equation}
and vary the diffusion parameter $\epsilon$ throughout the tests. Using the convection field $(a)$, we test the method's convergence properties for $\sigma$ given as in \eqref{sigma_rneqs} and $\mathcal{A}=\epsilon$; the  results are summarised in Figure \ref{fig:test-4}.

\begin{figure}[h!]
  \caption[]
  {\label{fig:test-4}
    %%%%%%%%%%%%%%%%%%%%%%%%%%%%%%%%%%%%%%%%%%%%%%%%%%%%%%%%%%%%%%%%%
    Convergence for a smooth solution of the convection-diffusion problem with coefficient $(a)$, for $\epsilon=10^{-1}$ and  $\epsilon=10^{-4}$.
    %%%%%%%%%%%%%%%%%%%%%%%%%%%%%%%%%%%%%%%%%%%%%%%%%%%%%%%%%%%%%%%%%% 
  }
  \begin{center}
        %%%%%%%%%%%%%%%%%%%%%%%%%%%%%%%%%%%%%%%%%%%%%%%%%%%%%%%%%%%%%%%%%%%% 
    \subfloat[{
        %%%%%%%%%%%%%%%%%%%%%%%%%%%%%%%%%%%%%%%%%%%%%%%%%%%%%%%%%%%%%%%
        {
          $\epsilon = 10^{-1}$
        }
      %%%%%%%%%%%%%%%%%%%%%%%%%%%%%%%%%%%%%%%%%%%%%%%%%%%%%%%%%%%%%%%% 
    }]{\hspace{-.5cm}
      \includegraphics[scale=\figscale,width=0.52\figwidth]{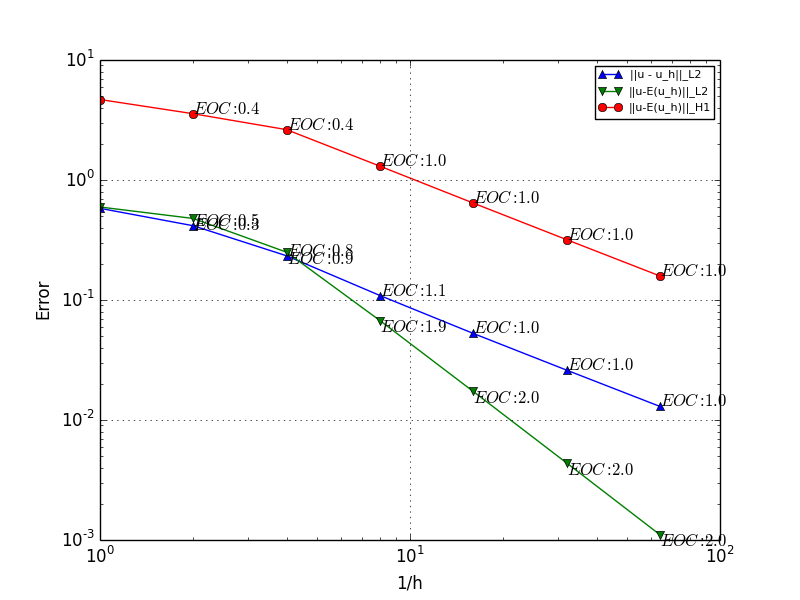}
    }    %%%%%%%%%%%%%%%%%%%%%%%%%%%%%%%%%%%%%%%%%%%%%%%%%%%%%%%%%%%%%%%%%%%
    \subfloat[{
      %%%%%%%%%%%%%%%%%%%%%%%%%%%%%%%%%%%%%%%%%%%%%%%%%%%%%%%%%%%%%%% 
        {
          $\epsilon = 10^{-4}$
        }
        %%%%%%%%%%%%%%%%%%%%%%%%%%%%%%%%%%%%%%%%%%%%%%%%%%%%%%%%%%%%%%%% 
    }]{
  \includegraphics[scale=\figscale,width=0.52\figwidth]{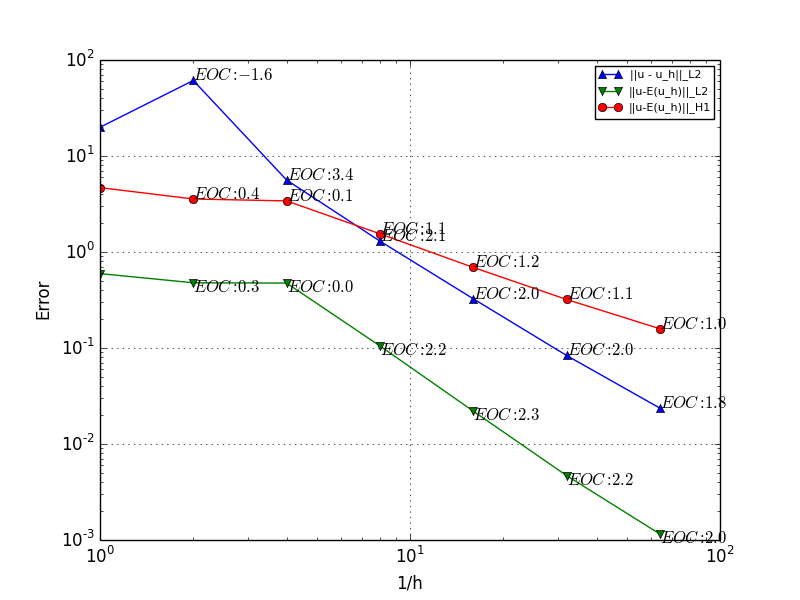}
    }
  \end{center}
  \vspace{-.6cm}
\end{figure}

To test the stability of the R-FEM approach, we consider \eqref{conv-diff} with the diagonal convection coefficient $(b)$ for small values of $\epsilon=10^{-2}, 10^{-3}$ on a quasi-uniform mesh with $h\approx 5\times 10^{-3}$. The exact solution to this problem admits a boundary layer of scale $\epsilon$ on the top and right boundaries. Therefore, the mesh is sufficiently fine to resolve the layer for $\epsilon=10^{-2}$ but not for $\epsilon=10^{-3}$. We compare the performance of R-FEM against a standard stable method, namely the (upwinded) interior penalty dG method which reads: find $u_h \in V^1_h$ such that
\begin{equation}\label{eq:dg-upwind}
  \begin{split}
    & \sum_{T\in\mathcal{T}}\Big(\int_T \epsilon \nabla u_h \cdot \nabla v_h + \vec w\cdot \nabla u_h  v_h +c u_h v_h\ud x+\int_{\partial_- T}( \vec w \cdot \vec n )\ujump{u_h} v_h^+\ud s \Big) \\
    &
    - \int_\Gamma \jump{u_h}\cdot \{\epsilon\nabla v_h \} + \jump{v_h}\cdot \{\epsilon\nabla u_h\}\ud s+\int_\Gamma \sigma  \jump{u_h} \cdot \jump{v_h} \ud s
    =
    \int_\Omega v_h \Foreach v_h \in V^1_h,
  \end{split}
\end{equation}
with $\sigma$ given by \eqref{sigma_reqs}.
Some representative solutions are given in Figure \ref{fig:ip}, using the recovered finite element method \eqref{eq:rem-upwind} (left plots) and the dG method \eqref{eq:dg-upwind} (right plots), respectively. We observe that R-FEM produces stable solution even in the non-resolving regime $\epsilon=10^{-4}$ including smooth profiles on the boundary layer. On the other hand, as expected, the dG solution is stable away from the layer and oscillates in the vicinity of the boundary layer. We can see, therefore, that R-FEM has the potential in delivering stable, conforming discretisations for stiff PDE problems.
\begin{figure}[]
  \caption[]
  {\label{fig:ip}
    %%%%%%%%%%%%%%%%%%%%%%%%%%%%%%%%%%%%%%%%%%%%%%%%%%%%%%%%%%%%%%%%%
    R-FEM and dG solutions over a uniform grid with $h\approx 0.005$ for two values of $\epsilon$.
    %%%%%%%%%%%%%%%%%%%%%%%%%%%%%%%%%%%%%%%%%%%%%%%%%%%%%%%%%%%%%%%%%% 
  }
  \begin{center}
    %%%%%%%%%%%%%%%%%%%%%%%%%%%%%%%%%%%%%%%%%%%%%%%%%%%%%%%%%%%%%%%%%%%% 
    \subfloat[{\label{fig:1b}
        %%%%%%%%%%%%%%%%%%%%%%%%%%%%%%%%%%%%%%%%%%%%%%%%%%%%%%%%%%%%%%% 
        {
          The R-FEM solution $\cE(u_h)$ with $\epsilon = 0.01$.
        }
        %%%%%%%%%%%%%%%%%%%%%%%%%%%%%%%%%%%%%%%%%%%%%%%%%%%%%%%%%%%%%%%% 
    }]{\hspace{-1cm}
      \includegraphics[scale=.23]{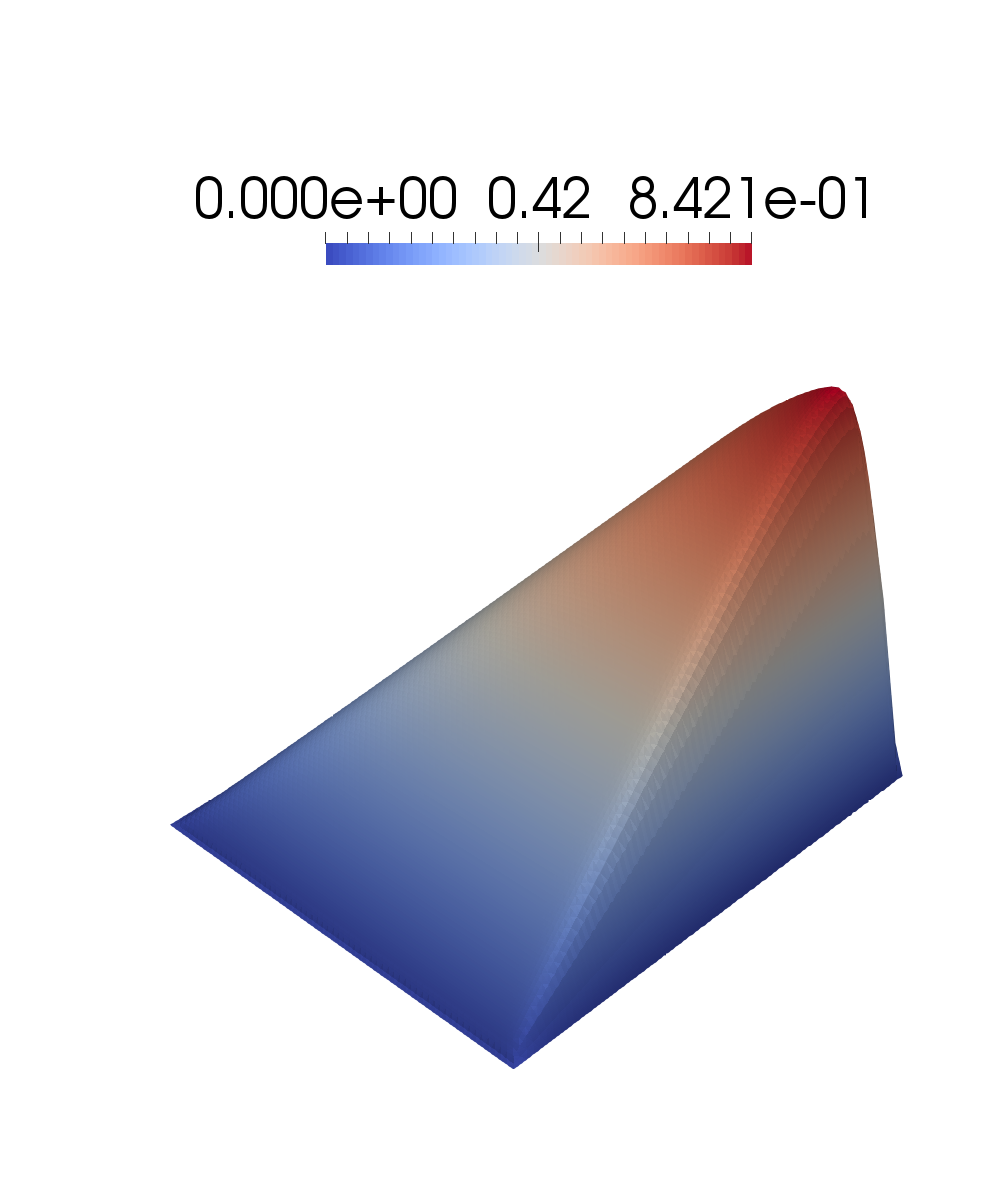}
    }
    %%%%%%%%%%%%%%%%%%%%%%%%%%%%%%%%%%%%%%%%%%%%%%%%%%%%%%%%%%%%%%%%%%%% 
    \subfloat[{\label{fig:1b}
        %%%%%%%%%%%%%%%%%%%%%%%%%%%%%%%%%%%%%%%%%%%%%%%%%%%%%%%%%%%%%%% 
        {
          The IP dG solution $u_h$ with $\epsilon = 0.01$.
        }
        %%%%%%%%%%%%%%%%%%%%%%%%%%%%%%%%%%%%%%%%%%%%%%%%%%%%%%%%%%%%%%%% 
    }]{
      \includegraphics[scale=.23]{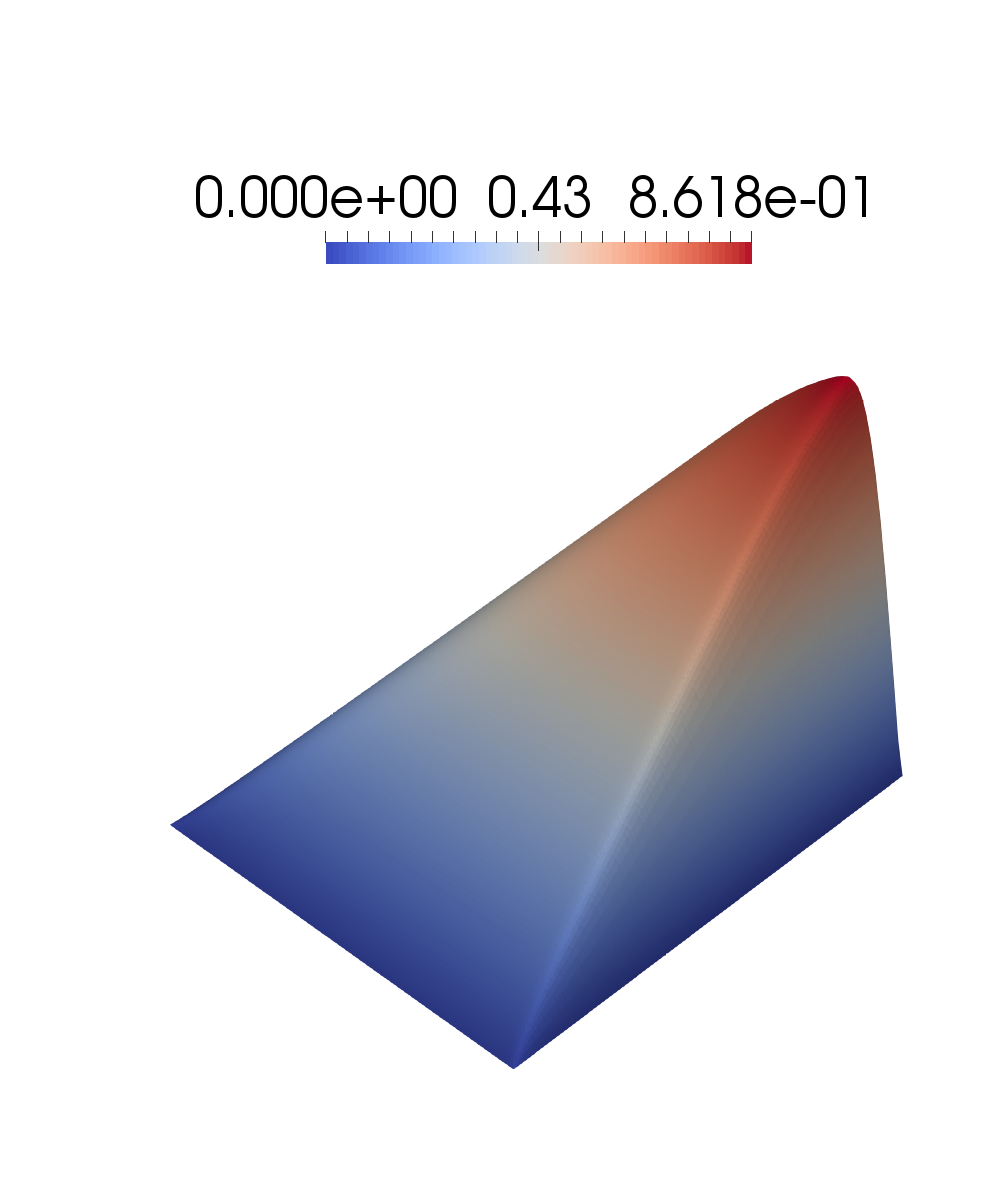}
    }
    \\
    %%%%%%%%%%%%%%%%%%%%%%%%%%%%%%%%%%%%%%%%%%%%%%%%%%%%%%%%%%%%%%%%%%%% 
    \subfloat[{\label{fig:1b}
        %%%%%%%%%%%%%%%%%%%%%%%%%%%%%%%%%%%%%%%%%%%%%%%%%%%%%%%%%%%%%%% 
        {
          The R-FEM solution $\cE(u_h)$ with $\epsilon = 0.001$.
        }
        %%%%%%%%%%%%%%%%%%%%%%%%%%%%%%%%%%%%%%%%%%%%%%%%%%%%%%%%%%%%%%%% 
    }]{\hspace{-1cm}
      \includegraphics[scale=.23]{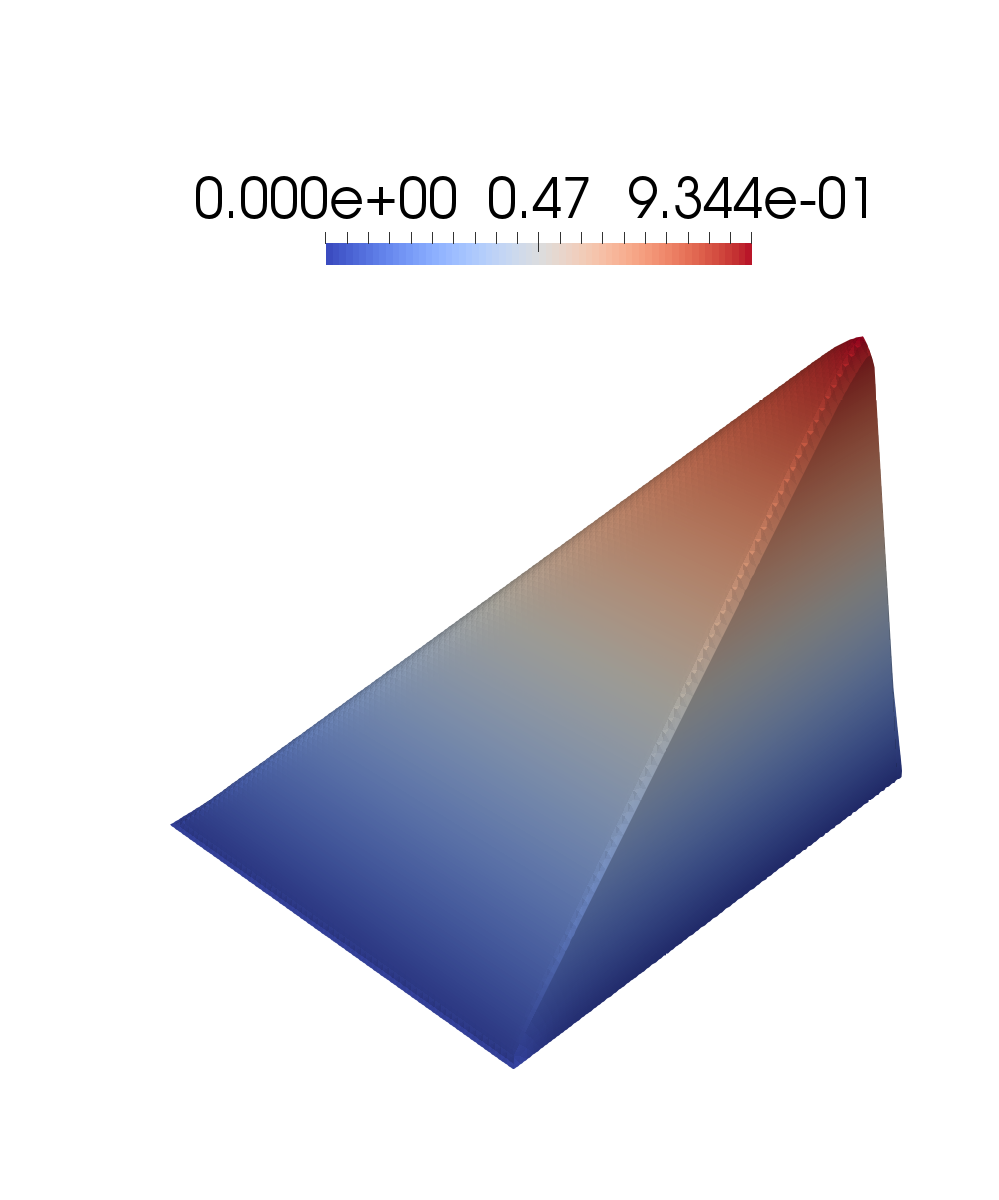}
    }
    %%%%%%%%%%%%%%%%%%%%%%%%%%%%%%%%%%%%%%%%%%%%%%%%%%%%%%%%%%%%%%%%%%%% 
    \subfloat[{\label{fig:1b}
      %%%%%%%%%%%%%%%%%%%%%%%%%%%%%%%%%%%%%%%%%%%%%%%%%%%%%%%%%%%%%%% 
        {
          The IP dG solution $u_h$ with $\epsilon = 0.001$.
        }
        %%%%%%%%%%%%%%%%%%%%%%%%%%%%%%%%%%%%%%%%%%%%%%%%%%%%%%%%%%%%%%%% 
    }]{
      \includegraphics[scale=.23]{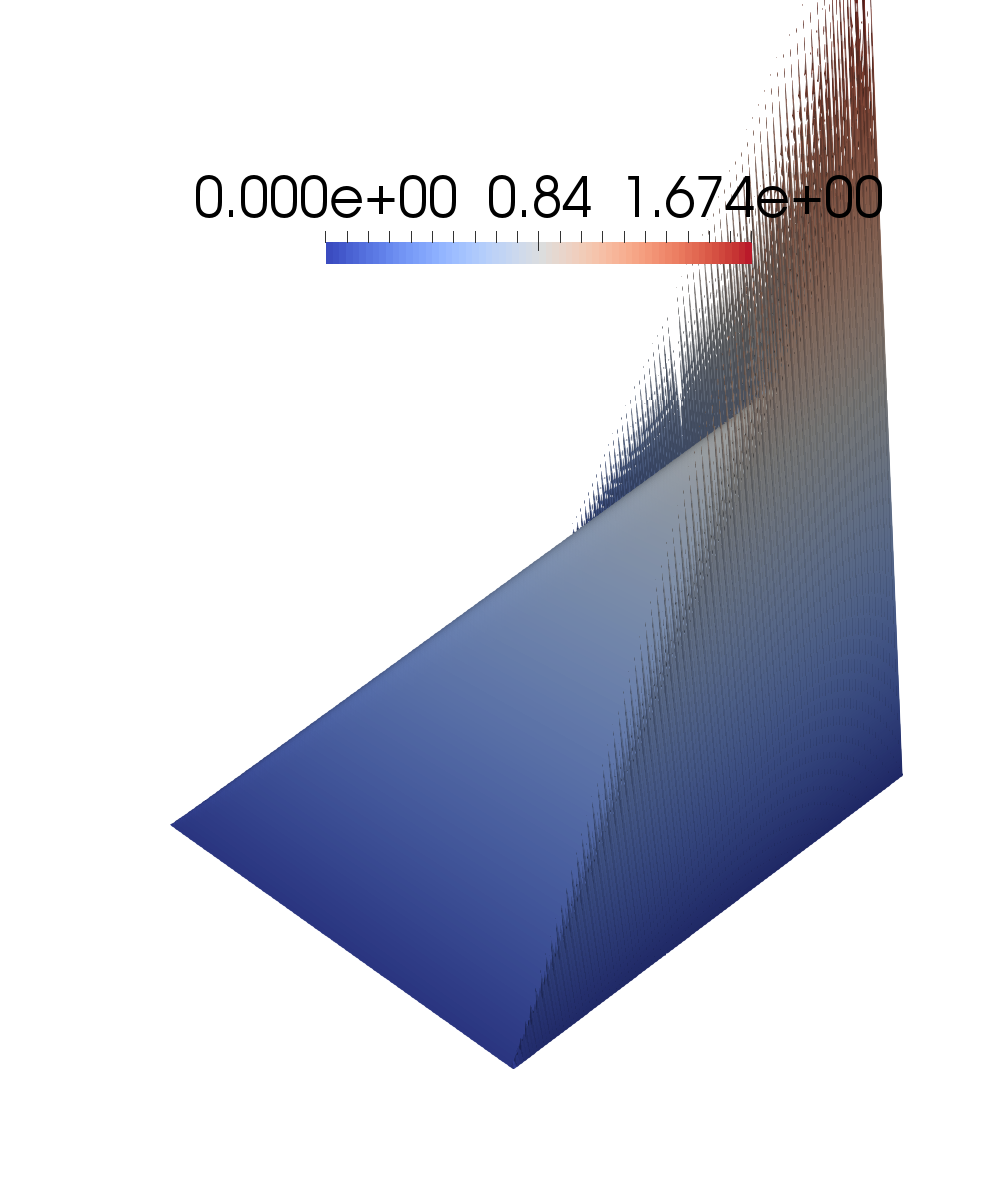}
    }
  \end{center}
\end{figure}

%% \clearpage

\section{Outlook and further applications}
 
In an effort to combine the advantages of discontinuous finite element spaces, yet produce conforming or classical non-conforming discretisations, we have introduced the framework of recovered finite element methods. We have demonstrated the convergence properties of R-FEM on standard simplicial and, in some instances, also on box-type meshes, we have proven an posteriori error bound and we have assessed the conditioning of the resulting stiffness matrices. As a first work on describing the R-FEM framework, we have refrained from describing advanced potential areas of applications in detail, in an effort to focus on the key ideas.  Nonetheless, R-FEM is envisaged to be successfully applicable to a number of directions. Here, we shall briefly comment on some of these.

\subsection{R-FEM for polytopic meshes}\label{ext:poly}
An important attribute of the definition of R-FEM \eqref{rem} is that the spaces $V_h^r$ and $V_h^s$ can differ. This idea has been used above in the context of recovering into different local polynomial degrees, yet on the \emph{same} mesh, aiming to improve on the convergence rate and/or the error per degree of freedom compared to respective conforming FEMs. However, it is by all means possible to consider spaces $V_h^r$ and $V_{\tilde{h}}^s$ subordinate to \emph{different} meshes. One natural application of this idea is the construction of R-FEM on general \emph{polygonal/polyhedral} (termed collectively henceforth as \emph{polytopic}) meshes, based on a recovery operator $\mathcal{E}:V_h^r\to V_{\tilde{h}}^s\cap H^1_0(\Omega)$, whereby $V_h^r$ is a discontinuous space subordinate to a polytopic mesh and $V_{\tilde{h}}^s\cap H^1_0(\Omega)$ is a conforming space subordinate to a (perhaps simplicial) finer mesh. This way once can construct \emph{conforming} approximations of elliptic problems on polytopic meshes where one has access to the whole of the approximate solution; this is in contrast to the recent virtual element framework \cite{virtual} whereby only certain functionals of the approximate solution are available to the user. This is a significant topic in its own right and will be discussed in detail in \cite{REMpoly}.

\subsection{R-FEM for high order PDEs}
The availability of classical $H^2$-conforming (e.g., Argyris or Hsieh-Clough-Tocher elements \cite{ciarlet})) has allowed for the construction of respective recovery operators $\mathcal{E}(w)\in H^2(\Omega)$ for $w\in V_h^r$ or $w\in V_h^r\cap H^1(\Omega)$ \cite{brenner,GHV}. Therefore, it is possible to construct R-FEM methods for biharmonic/plate problems. As this is also a  significant topic in its own right, it will be discussed elsewhere.

\subsection{Mixed R-FEM for problems with constraints}
Upon constructing suitable recovery operators, it is interesting to revisit the classical problem of numerical approximation of PDE problems with constraints, e.g., (nearly) incompressible elasticity, Darcy or viscous incompressible flows. Indeed, the flexibility in the underlying degrees of freedom offered by R-FEM in conjunction with suitable recovery operators, e.g., recovering onto classical inf-sup stable mixed finite element pairs is a potentially substantial direction of further research. For some first results in this context, we refer to \cite{REMmixed}.

\section*{Acknowledgements} We wish to express our sincere gratitude to Gabriel R.~Barrenechea (University of Strathclyde), Andrea Cangiani and Zhaonan (Peter) Dong (University of Leicester) for their insightful comments on an earlier version of this work, which lead to a clearer understanding of the method. EHG acknowledges funding by The Leverhulme Trust.

\bibliographystyle{siam}
\bibliography{bibliography}

\newcommand{\noop}[1]{}
\begin{thebibliography}{10}

\bibitem{arnold}
{\sc D.~N. Arnold}, {\em An interior penalty finite element method with
  discontinuous elements}, SIAM J. Numer. Anal., 19 (1982), pp.~742--760.

\bibitem{unified}
{\sc D.~N. Arnold, F.~Brezzi, B.~Cockburn, and L.~D. Marini}, {\em Unified
  analysis of discontinuous {G}alerkin methods for elliptic problems}, SIAM J.
  Numer. Anal., 39 (2001/02), pp.~1749--1779.

\bibitem{bs}
{\sc R.~E. Bank and L.~R. Scott}, {\em On the conditioning of finite element
  equations with highly refined meshes}, SIAM J. Numer. Anal., 26 (1989),
  pp.~1383--1394.

\bibitem{REMmixed}
{\sc G.~R. Barrenechea, E.~H. Georgoulis, and T.~Pryer}, {\em Recovered mixed
  finite element methods}, in preparation.

\bibitem{virtual}
{\sc L.~Beir\~ao~da Veiga, F.~Brezzi, A.~Cangiani, G.~Manzini, L.~D. Marini,
  and A.~Russo}, {\em Basic principles of virtual element methods}, Math.
  Models Methods Appl. Sci., 23 (2013), pp.~199--214.

\bibitem{MR1140646}
{\sc S.~C. Brenner and L.-Y. Sung}, {\em Linear finite element methods for
  planar linear elasticity}, Math. Comp., 59 (1992), pp.~321--338.

\bibitem{brenner}
\leavevmode\vrule height 2pt depth -1.6pt width 23pt, {\em {$C^0$} interior
  penalty methods for fourth order elliptic boundary value problems on
  polygonal domains}, J. Sci. Comput., 22/23 (2005), pp.~83--118.

\bibitem{BuffaOrtner}
{\sc A.~Buffa and C.~Ortner}, {\em Compact embeddings of broken {S}obolev
  spaces and applications}, IMA J. Numer. Anal., 29 (2009), pp.~827--855.

\bibitem{DGpolyparabolic}
{\sc A.~Cangiani, Z.~Dong, and E.~H. Georgoulis}, {\em {$hp$}-version
  space-time discontinuous {G}alerkin methods for parabolic problems on
  prismatic meshes}, SIAM J. Sci. Comput., to appear.

\bibitem{DGpoly2}
{\sc A.~Cangiani, Z.~Dong, E.~H. Georgoulis, and P.~Houston}, {\em
  {$hp$}-version discontinuous {G}alerkin methods for
  advection-diffusion-reaction problems on polytopic meshes}, ESAIM Math.
  Model. Numer. Anal., 50 (2016), pp.~699--725.

\bibitem{DGpoly1}
{\sc A.~Cangiani, E.~H. Georgoulis, and P.~Houston}, {\em {$hp$}-version
  discontinuous {G}alerkin methods on polygonal and polyhedral meshes}, Math.
  Models Methods Appl. Sci., 24 (2014), pp.~2009--2041.

\bibitem{ciarlet}
{\sc P.~G. Ciarlet}, {\em The finite element method for elliptic problems},
  vol.~40 of Classics in Applied Mathematics, Society for Industrial and
  Applied Mathematics (SIAM), Philadelphia, PA, 2002.
\newblock Reprint of the 1978 original [North-Holland, Amsterdam; MR0520174 (58
  \#25001)].

\bibitem{MR0400739}
{\sc P.~Cl\'ement}, {\em Approximation by finite element functions using local
  regularization}, RAIRO Analyse Num{\'e}rique, 9 (1975), pp.~77--84.

\bibitem{HDG}
{\sc B.~Cockburn, J.~Gopalakrishnan, and R.~Lazarov}, {\em Unified
  hybridization of discontinuous {G}alerkin, mixed, and continuous {G}alerkin
  methods for second order elliptic problems}, SIAM J. Numer. Anal., 47 (2009),
  pp.~1319--1365.

\bibitem{MR92e:65128}
{\sc B.~Cockburn and C.-W. Shu}, {\em The {R}unge-{K}utta local projection
  ${P}\sp 1$-discontinuous-{G}alerkin finite element method for scalar
  conservation laws}, RAIRO Mod\'el. Math. Anal. Num\'er., 25 (1991),
  pp.~337--361.

\bibitem{Ern}
{\sc D.~A. Di~Pietro, A.~Ern, and S.~Lemaire}, {\em An arbitrary-order and
  compact-stencil discretization of diffusion on general meshes based on local
  reconstruction operators}, Comput. Methods Appl. Math., 14 (2014),
  pp.~461--472.

\bibitem{REMpoly}
{\sc Z.~Dong, E.~H. Georgoulis, and T.~Pryer}, {\em Recovered finite element
  methods on polygonal and polyhedral meshes}, in preparation.

\bibitem{Lew}
{\sc A.~T. Eyck and A.~Lew}, {\em Discontinuous {G}alerkin methods for
  non-linear elasticity}, Internat. J. Numer. Methods Engrg., 67 (2006),
  pp.~1204--1243.

\bibitem{GHV}
{\sc E.~H. Georgoulis, P.~Houston, and J.~Virtanen}, {\em An {\it a posteriori}
  error indicator for discontinuous {G}alerkin approximations of fourth-order
  elliptic problems}, IMA J. Numer. Anal., 31 (2011), pp.~281--298.

\bibitem{REMcode}
{\sc E.~H. Georgoulis and T.~Pryer}, {\em An implementation of the recovered
  finite element method}, in \url{http://dx.doi.org/10.5281/zenodo.572279},
  2017.

\bibitem{MR1886000}
{\sc P.~Hansbo and M.~G. Larson}, {\em Discontinuous {G}alerkin methods for
  incompressible and nearly incompressible elasticity by {N}itsche's method},
  Comput. Methods Appl. Mech. Engrg., 191 (2002), pp.~1895--1908.

\bibitem{hss}
{\sc P.~Houston, C.~Schwab, and E.~S\"uli}, {\em Discontinuous {$hp$}-finite
  element methods for advection-diffusion-reaction problems}, SIAM J. Numer.
  Anal., 39 (2002), pp.~2133--2163.

\bibitem{MR88b:65109}
{\sc C.~Johnson and J.~Pitk{\"a}ranta}, {\em An analysis of the discontinuous
  {G}alerkin method for a scalar hyperbolic equation}, 46 (1986), pp.~1--26.

\bibitem{KP}
{\sc O.~A. Karakashian and F.~Pascal}, {\em Convergence of adaptive
  discontinuous {G}alerkin approximations of second-order elliptic problems},
  SIAM J. Numer. Anal., 45 (2007), pp.~641--665 (electronic).

\bibitem{MR1343077}
{\sc R.~Kouhia and R.~Stenberg}, {\em A linear nonconforming finite element
  method for nearly incompressible elasticity and {S}tokes flow}, Comput.
  Methods Appl. Mech. Engrg., 124 (1995), pp.~195--212.

\bibitem{MR58:31918}
{\sc P.~Lesaint and P.-A. Raviart}, {\em On a finite element method for solving
  the neutron transport equation}, in Mathematical aspects of finite elements
  in partial differential equations (Proc. Sympos., Math. Res. Center, Univ.
  Wisconsin, Madison, Wis., 1974), Math. Res. Center, Univ. of
  Wisconsin-Madison, Academic Press, New York, 1974, pp.~89--123. Publication
  No. 33.

\bibitem{MR1248895}
{\sc P.~Oswald}, {\em On a {BPX}-preconditioner for {${\rm P}1$} elements},
  Computing, 51 (1993), pp.~125--133.

\bibitem{pryer}
{\sc T.~Pryer}, {\em An a posteriori analysis of some inconsistent,
  nonconforming galerkin methods approximating elliptic problems}, in Tech
  report on arXiv \url{https://arxiv.org/pdf/1505.04318}.

\bibitem{reedhill}
{\sc W.~Reed and T.~Hill}, {\em Triangular mesh methods for the neutron
  transport equation.}, Technical Report {\rm LA-UR-73-479} Los Alamos
  Scientific Laboratory,  (1973).

\bibitem{MR1011446}
{\sc L.~R. Scott and S.~Zhang}, {\em Finite element interpolation of nonsmooth
  functions satisfying boundary conditions}, Math. Comp., 54 (1990),
  pp.~483--493.

\bibitem{MR2027288}
{\sc T.~P. Wihler}, {\em Locking-free {DGFEM} for elasticity problems in
  polygons}, IMA J. Numer. Anal., 24 (2004), pp.~45--75.

\end{thebibliography}

\end{document}